\documentclass[letterpaper,11pt]{article}

\usepackage{amsmath,amsthm}
\usepackage{amsfonts}
\usepackage{latexsym}
\usepackage{amssymb,bm}
\usepackage{mathrsfs}
\usepackage{mathtools}
\usepackage[page]{appendix}
\usepackage{float}
\usepackage{changes}
\usepackage{placeins}
\usepackage{tabulary}
\usepackage{tikz}
\usepackage{subfig}
\usepackage{graphicx}
\usepackage{pgfplots}
\usepackage{xcolor}
\usepackage{graphicx}
\usepackage{epstopdf}
\usepackage{grffile}
\usepackage{verbatim}
\usepackage{array}
\usepackage{booktabs}
\usepackage{multirow}

\usetikzlibrary{calc,trees,positioning,arrows,chains,shapes.geometric,decorations.pathreplacing,decorations.pathmorphing,shapes,matrix,shapes.symbols}
\usetikzlibrary{shapes,backgrounds}
\pgfplotsset{width=13cm,compat=1.14}
\usepackage[all,knot,arc,import,poly]{xy}
\usepackage{array}
\usepackage{booktabs}

\usetikzlibrary{arrows.meta}
\pgfplotstableread[row sep=\\,col sep=&]{
    simname     &   PS    &   PSIE    \\
    BlogFeedback      &   2968  &   2342    \\
    Crime      &  211   &   216   \\
    DrivFace       &   2008  &  585   \\
    Mug32     &   203   &   192     \\
    Wisconsin     &   211   &   210      \\
    Random E        &   408   &   151     \\
    Random J       &   507  &   278 \\
}\dispdata
\usetikzlibrary{arrows.meta}
\pgfplotstableread[row sep=\\,col sep=&]{
    simname_time    &   PS_time   &   PSIE_time    \\
    BlogFeedback      &   207.18   &  130.44   \\
    DrivFace       &   133.19   &   37.11    \\
    Random E        &   10.53   &   4.08 \\
    Random J     &   20.09   &   11.31     \\
}\dispdata
\usetikzlibrary{arrows.meta}
\pgfplotstableread[row sep=\\,col sep=&]{
    simname_time1    &   PS_time1   &   PSIE_time1    \\
    Crime      &   0.85   &   0.78   \\
    Mug32     &   1.36  &   1.18     \\
    Wisconsin     &  0.15   &  0.11     \\
}\dispdata

\usepackage{changes}

\usepackage{color}
\usepackage{color}

\newtheorem{theorem}{Theorem}[section]
\newtheorem{lemma}[theorem]{Lemma}

\newtheorem{proposition}[theorem]{Proposition}

\newtheorem{algorithm}{Algorithm}

\newcommand{\inner}[2]{\langle #1,#2\rangle}

\newcommand{\norm}[1]{\|{#1}\|}

\newcommand{\R}{\mathbb{R}}

\newcommand{\comenta}[1]{}

\newcommand{\HH}{\mathcal{H}}

\newcommand{\mgap}{\vspace{.1in}}

\begin{document}
\title{A relative-error inertial-relaxed inexact projective splitting algorithm}
\author{
    M. Marques Alves
\thanks{
Departamento de Matem\'atica,
Universidade Federal de Santa Catarina,
Florian\'opolis, Brazil, 88040-900 ({\tt maicon.alves@ufsc.br}).
The work of this author was partially supported by CNPq grant no. 308036/2021-2.}
\and
  Marina Geremia
\thanks{
Departamento de Matem\'atica,
Universidade Federal de Santa Catarina,
Florian\'opolis, Brazil, 88040-900.  Departamento de Ensino, Pesquisa e Extensão, Instituto Federal de 
Santa Catarina (IFSC) ({\tt marina.geremia@ifsc.edu.br}).}
\and
Raul T. Marcavillaca
\thanks{
Departamento de Matemáticas, Universidad de Tarapacá, Arica, Chile ({\tt raultm.rt@gmail.com}).
The work of this author was partially supported by CAPES.}
}


\maketitle

\begin{abstract}
For solving structured monotone inclusion problems involving the sum of finitely many maximal monotone operators, 
we propose and study a relative-error inertial-relaxed inexact projective splitting algorithm. The proposed algorithm benefits from a combination of inertial and relaxation effects, which are both controlled by parameters within a certain range. We propose sufficient conditions on these parameters and study the interplay between them in order to guarantee weak convergence of sequences generated by our algorithm. Additionally, the proposed algorithm also benefits from inexact subproblem solution within a relative-error criterion. 
Illustrative numerical experiments on LASSO problems indicate some improvement when compared with 
previous (noninertial and exact) versions of projective splitting. 
\end{abstract}

\pagestyle{plain}

\section{Introduction}
\label{sec:int}
Let $\HH_0, \HH_1,\ldots, \HH_n$ be real Hilbert spaces and let $\inner{\cdot}{\cdot}$ and
 $\|\cdot\|=\sqrt{\inner{\cdot}{\cdot}}$ denote the inner product and norm (respectively) in $\HH_i$ ($i=0,\dots, n$).
 Assume that $\HH_0=\HH_n$. 
 Let $\boldsymbol{\HH}:=\HH_0\times \dots\times \HH_{n-1}$ be endowed with the inner product and norm
defined, respectively, as follows (for some $\gamma >0$):
\begin{align}
  \label{eq:def.inner}
\inner{(z,w)}{(z',w')}_{\gamma}=\gamma\inner{z}{z'}+\sum_{i=1}^{n-1}\inner{w_i}{w'_i},\quad
\norm{(z,w)}_\gamma^2=\gamma \norm{z}^2+\sum_{i=1}^{n-1}\norm{w_i}^2,
\end{align}
where $z,z'\in \HH_0$ and $w:=(w_1,\dots, w_{n-1}), w':=(w'_1,\dots, w'_{n-1})\in \HH_1\times \ldots\times \HH_{n-1}$.

Consider the monotone inclusion problem of finding $z\in \HH_0$ such that
\begin{align}
 \label{eq:Pmip}
 0\in \sum_{i=1}^{n}G_i^*T_iG_i(z)
\end{align}
where $n\geq 2$ and the following assumptions hold:
\begin{itemize}
  \item[\mbox{(A1)}] For each $i=1,\dots, n$, the operator $T_i:\HH_i\rightrightarrows \HH_i$ is (set-valued) maximal 
monotone and $G_i:\HH_0\rightarrow \HH_i$  is a bounded linear operator.
  \item[\mbox{(A2)}] The linear operator $G_n$ is equal to the identity map in $\HH_0=\HH_n$, i.e., $G_n:z\mapsto z$ 
for all $z\in \HH_0$.
  \item[\mbox{(A3)}] The solution set of \eqref{eq:Pmip} is nonempty, i.e., there exists at least one $z\in \HH_0$ satisfying
the inclusion in \eqref{eq:Pmip}.
\end{itemize}

Problem \eqref{eq:Pmip} appears in different fields of applied mathematics and optimization, 
including machine learning, inverse problems and image processing~\cite{alot.combet.shah.2014,Johnst.Eckst.2018,Johnst.Eckst.2019}, specially in connection with
the convex optimization problem
\begin{align}
  \label{eq:opt}
 \min_{z\in \HH_0}\,\sum_{i=1}^n\,f_i(G_i z)
\end{align}
where, for $i=1,\dots, n$, each $f_i:\HH_i\to (-\infty,\infty]$ is proper, lower semicontinuous and convex. Indeed, under mild assumptions
on $f_i$ and $G_i$, the minimization problem \eqref{eq:opt} is equivalent to the monotone inclusion problem \eqref{eq:Pmip} with $T_i=\partial f_i$ ($i=1,\dots, n$). 

A very popular strategy to find approximate solutions of \eqref{eq:Pmip} is that of (monotone) operator splitting algorithms, 
which traces back to the development of some well-known numerical schemes like the Douglas-Rachford splitting algorithm, Spingarn's method of partial inverses, among others. 

The family of \emph{projective splitting algorithms} for solving \eqref{eq:Pmip} originated in \cite{eck.sva-gen.sjco09}
for the case when $G_i$ is the identity ($i=1,\dots, n$), and later on it was developed in different directions (see, e.g., 
\cite{
combettes2015BestAppr,
alot.combet.shah.2014,
combettes2020warped,
combettes2022saddle,
combettes2018asynchronous,
combettes2016Solving,
eckstein2019rates,
Johnst.Eckst.2019,
Johnst.Eckst.2018,
jonhst.eckst.siam2019}). 
It has deserved a lot of attention
in modern operator splitting research, mainly due to its flexibility (when compared to other classes of operator splitting algorithms) regarding parameters and the activation of $T_i$ and $G_i$ separately during the iterative process.

The derivation of the class of projective splitting algorithms can be motivated as follows.
First note that using Assumption (A2) above, we obtain that \eqref{eq:Pmip} can equivalently be written as
\begin{align}
 \label{eq:Pmip02}
 0\in \sum_{i=1}^{n-1}G_i^*T_iG_i(z) + T_n(z)
\end{align}
which, in turn, is clearly equivalent to the (feasibility) problem of finding a point in the \emph{extended solution set} of \eqref{eq:Pmip} 
(or \eqref{eq:Pmip02}):
\begin{align}
\label{eq:KTSi}
  \mathcal{S}:=\left\{(z,w_1,\ldots,w_{n-1})\in \boldsymbol{\HH}\;\;|\;\;w_i\in T_i(G_iz),\, i=1,\ldots,n-1,\, 
-\sum_{i=1}^{n-1}G^*_iw_i\in T_n(z)\right\}.
\end{align}
Since $\mathcal{S}$ is nonempty (see Assumption (A3)), closed and convex 
(see, e.g., \cite{alot.combet.shah.2014,eck.sva-gen.sjco09}) in   
$\boldsymbol{\HH}$, it follows that problem \eqref{eq:Pmip} reduces to the task of finding a point in $\mathcal{S}$ (fact that
motivates the abstract framework developed in Section \ref{sec:spm} below).

Note now that, if we pick $y_i^k\in T_i(x_i^k)$ ($i=1,\dots, n$), then, from the monotonicity of $T_i$ and the inclusions
in \eqref{eq:KTSi}, it follows that
\begin{align}
  \label{eq:sep.i}
 \sum_{i=1}^n\,\inner{G_i z-x_i^k}{y_i^k- w_i}\leq 0\qquad \forall (z,w_1,\dots, w_{n-1})\in \mathcal{S},
\end{align}
where 
\begin{align}
 \label{eq:wni}
 w_n:=-\sum_{i=1}^{n-1}G^*_iw_i.
 \end{align}
The inequality \eqref{eq:sep.i} says, in particular, that $\{(x_i^k,y_i^k)\}_{i=1}^n$  defines a function of $(z,w_1,\dots, w_{n-1})$ which is lesser or equal to zero in $\mathcal{S}$. 
Since this function can be proved to be affine (see, e.g., Lemma \ref{lmm:grad.sep} below), it follows 
from \eqref{eq:sep.i} that it defines a semispace in $\boldsymbol{\HH}$ containing the extended solution set $\mathcal{S}$.

Consequently, it follows that the main mechanism behind the idea of projective splitting algorithms is basically as follows: 
at the iteration $(z^k,w_1^k,\dots, w_{n-1}^k)$, pick, for each $i=1,\dots, n$, a pair $(x_i^k,y_i^k)$ in the graph of $T_i$ and
then update the current iterate $p^k:=(z^k,w_1^k,\dots, w_{n-1}^k)$ to 
$p^{k+1}:=(z^{k+1},w_1^{k+1},\dots, w_{n-1}^{k+1})$ by projecting $p^k$ onto the semispace defined by the affine function given in the left hand side of \eqref{eq:sep.i}. 
Computation of $(x_i^k,y_i^k)$ is in general performed by (inexactly) activating the resolvent $(T_i+I)^{-1}$ operator of each $T_i$ to guarantee, in particular, that
the current iterate $(z^k,w_1^k,\dots, w_{n-1}^k)$ belongs to the positive side of the corresponding hyperplane.

%

\mgap

In this paper, we propose and study a relative-error inertial-relaxed inexact projective splitting algorithm for solving \eqref{eq:Pmip} and, in particular, for solving the convex program \eqref{eq:opt}. Inertial algorithms for solving monotone inclusions of the form $0\in T(z)$, where $T$ is maximal monotone, 
%
were first proposed in \cite{alv.att-iner.svva01}, and since then developed by different authors and in different directions of 
research (see, e.g., \cite{alv.eck.ger.mel-preprint19,att.pey-con.mp19,bot.cse.hen-ine.amc15,com.gla-qua.sjo17} and references there in). 
At a current iterate, say $p^k$, the inertial effect in the iterative process is produced by an extrapolation step of the form 
(see also Algorithm \ref{inertial_projective} and Figure \ref{fig:arrows} below):
\[
 \widehat p^k = p^k + \alpha_k(p^k - p^{k-1}).
\]
Since $\alpha_k\geq 0$ controls the magnitude of extrapolation performed in the direction of the vector $p^k-p^{k-1}$, it follows
that the asymptotic behavior and size of $\alpha_k$ have a direct influence in the convergence analysis of inertial-type algorithms. 
A usual sufficient condition \cite{alv.att-iner.svva01} imposed on the sequence $\{\alpha_k\}$, with guarantee of
weak convergence 
of $\{p^k\}$, is that
$\{\alpha_k\}$ is nondecreasing and $\alpha_k<1/3$ for all $k\geq 0$. The upper bound $1/3$ has been recently improved in
combination with relaxation effects~\cite{alv.eck.ger.mel-preprint19,att.cab-con.pre18}.

The main goal of this paper is to develop a projective splitting-type algorithm for solving \eqref{eq:Pmip} with
both inertial and relaxation effects and, additionally, with inexact subproblems solution within relative-error criterion.
 Up to the authors knowledge, this is the first time in the literature that
inertial effects are considered in projective splitting algorithms. Our main algorithm is Algorithm \ref{alg:iPSM} from
Section \ref{sec:ps}, for which the convergence is studied in Theorems \ref{th:conv01} and \ref{th:conv02}, under flexible
assumptions on the inertial and relaxation parameters. Motivated by the discussion above that \eqref{eq:Pmip} is equivalent
to the problem of finding a point in the closed and convex set $\mathcal{S}$ as in \eqref{eq:KTSi}, we first introduce in Section 
\ref{sec:spm}  an inertial-relaxed separator-projector method for solving the (feasibility) problem of finding points in closed
convex subsets of Hilbert spaces.

\mgap

The following well-known property (see, e.g., \cite[Corollary 2.14]{bau.com-book}) will be useful in this paper: for all $x,y$ in a real Hilbert space $\HH$ and $t\in \R$,
 it holds that
\begin{align}
 \label{eq:ineq.s}
 \norm{tx+(1-t)y}^2=t\norm{x}^2+(1-t)\norm{y}^2-t(1-t)\norm{x-y}^2.
 \end{align}
We shall also use the following inequality:
\begin{align}
  \label{eq:ineq.sum2}
 \left\|\sum_{i=1}^n\,x_i\right\|^2\leq n\sum_{i=1}^n\,\norm{x_i}^2.
\end{align}

%
\section{An inertial-relaxed separator-projection method}
 \label{sec:spm}

In this section, we propose and study a general separator-projection framework (Algorithm \ref{inertial_projective}) for 
finding a point in a closed and convex subset of a Hilbert space. 
The main motivation comes from the fact (as previously discussed in Section \ref{sec:int}) that the monotone inclusion 
problem \eqref{eq:Pmip} can be reformulated as the problem of finding a point in the extended solution set $\mathcal{S}$ as in \eqref{eq:KTSi}.  
Algorithm \ref{inertial_projective} will be used in Section \ref{sec:ps} to analyze the convergence of the main algorithm proposed in 
this paper (namely Algorithm \ref{alg:iPSM}) for solving \eqref{eq:Pmip1}. 

Let $\HH$ be a real Hilbert space with inner product $\inner{\cdot}{\cdot}$
and norm $\|\cdot\|=\sqrt{\inner{\cdot}{\cdot}}$. We denote the gradient of an affine
function $\varphi:\HH\to \R$ by the usual notation $\nabla \varphi$ and, in this case, we
also write $\varphi(z)=\inner{\nabla \varphi}{z}+\varphi(0)$ for all $z\in \HH$.

\mgap
\mgap

\noindent
\fbox{
\addtolength{\linewidth}{-2\fboxsep}%
\addtolength{\linewidth}{-2\fboxrule}%
\begin{minipage}{\linewidth}
\begin{algorithm}
\label{inertial_projective}
{\bf An inertial-relaxed linear separator-projection method for finding a point in a nonempty closed convex set $\mathcal{S}\subset \mathcal{H}$}
\end{algorithm}
\begin{itemize}
\item[{\bf(0)}] Let  $p^0=p^{-1}\in \HH$, $\alpha\in [0,1)$ and $0<\underline{\beta}<\overline{\beta}<2$ be given and let $k\leftarrow0$.
\item [{\bf(1)}] Choose $\alpha_{k}\in [0,\alpha]$ and define
  \begin{align}
      \label{eq:ext}
     \widehat p^{\,k}= p^{k}+\alpha_{k}(p^{k}-p^{k-1}).
 \end{align}
\item [{\bf(2)}] Find an affine function $\varphi_k$ such that $\nabla \varphi_k\neq 0$ and $\varphi_k(p)\leq 0$ for all $p\in \mathcal{S}$.
%
Choose $\beta_k\in [\underline{\beta},\overline{\beta}]$ and set
 \begin{align}
  \label{eq:ext.proj}
   p^{k+1}=\widehat p^{\,k} - \dfrac{\beta_k \max\{0,\varphi_k(\widehat p^{\,k})\}}{\norm{\nabla \varphi_k}^2}\nabla \varphi_k.
 \end{align}
\item[{\bf(3)}] Let $k\leftarrow k+1$ and go to step 1.
\end{itemize}
\noindent
\end{minipage}
} 

\mgap
\mgap
\noindent
{\bf Remarks}.
\begin{itemize}
\item[\mbox{(i)}] Letting $\widetilde p^{\,k+1}$ be the (orthogonal) projection of $\widehat p^{\,k}$ onto the semispace
$\{p\in \HH\,|\,\varphi_k(p)\leq 0\}$, i.e.,
\begin{align}
 \label{eq:001}
 \widetilde p^{\,k+1} = \widehat p^{\,k} - \dfrac{\max\{0,\varphi_k(\widehat p^{\,k})\}}
{\norm{\nabla \varphi_k}^2}\nabla \varphi_k
\end{align}
and using \eqref{eq:ext.proj} we conclude that
\begin{align}
  \label{eq:333}
 p^{k+1}
 = \widehat p^{\,k} + \beta_k (\widetilde p^{\,k+1} - \widehat p^{\,k}).
\end{align}
\item[\mbox{(ii)}] Note that \eqref{eq:ext} and \eqref{eq:333} illustrate the different effects promoted in 
Algorithm \ref{inertial_projective} by inertia and relaxation, which are controlled, respectively, by the parameters $\alpha_k$ and
$\beta_k$. See Figure \ref{fig:arrows} below.
\begin{figure}[!htb]
\centering
\vspace{0.5cm}
\begin{tikzpicture}[scale=3]
\draw[thick,->,blue] (0, 0) -- (1.5, 0.3);
\draw[thick,->,green] (1.5, 0.3) -- (2, 0.75);
\draw[dashed] (1.1,1.4) -- (2.40,-0.1);
\filldraw (0,0) circle (0.6pt) node[align=center,   below] {$p^{k-1}$};
\filldraw (1,0.2) circle (0.6pt) node[align=center, below] {$p^{k}$};
\filldraw (1.5,0.3) circle (0.5pt) node[align=center, below] {$\widehat p^k$};
\filldraw (2,0.75) circle (0.5pt) node[align=center, right] {$p^{k+1}$};
\draw (2.55,0) circle (0.0pt) node[align=center, below] {$\{p\in \mathcal{H}\;|\;\varphi_k(p)=0\}$};
\draw[thick,->,blue] (1,0.2) -- (2.5,1.025);
\filldraw (2.5,1.025) circle (0.6pt) node[align=center, right] {$\widehat p^{k+1}$};
\end{tikzpicture}
\vspace{-0.5ex}
\caption{Geometric interpretation of steps~\eqref{eq:ext}
and~\eqref{eq:ext.proj} in Algorithm~\ref{inertial_projective}.  The (overrelaxed)
projection step~\eqref{eq:ext.proj} is orthogonal to the separating hyperplane
$\{p\in \mathcal{H}\;|\;\varphi_k(p)=0\}$, which can differ from the direction between $p^{k-1}$, $p^k$,
and $\widehat p^k$ when $\alpha_k > 0$.}\label{fig:arrows}
\end{figure}
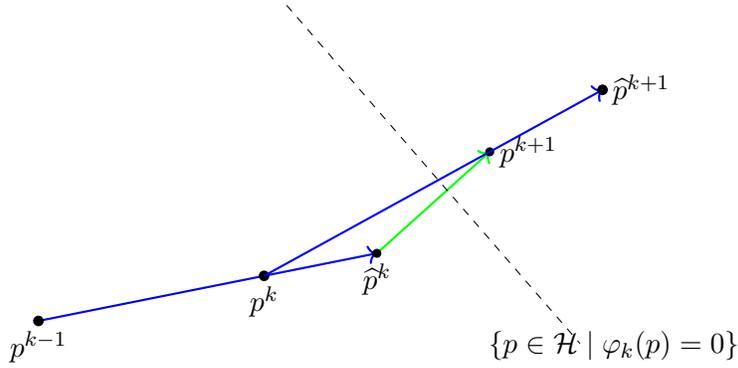
\item[\mbox{(iii)}] If $\alpha_k\equiv 0$, in which case $\widehat p^k=p^k$ in \eqref{eq:ext}, then it follows that Algorithm 
\ref{inertial_projective} reduces to the well-known linear separator-projection method for finding a point in 
$\mathcal{S}\subset \HH$ (see, e.g., \cite{alot.combet.shah.2014}).
\item[\mbox{(iv)}] As we mentioned early, Algorithm \ref{inertial_projective} will be used in the next section for analyzing the
convergence of Algorithm \ref{alg:iPSM}. The main convergence results for Algorithm \ref{inertial_projective} will be stated in 
this section, in Theorems \ref{lm:conv.proj} and \ref{thm:second.conv} below.
\item[\mbox{(v)}] The construction of an affine function $\varphi_k(\cdot)$ as in step 2 of Algorithm \ref{inertial_projective} in general depends on the particular structure of the set $\mathcal{S}$. For example, in Algorithm \ref{alg:iPSM} below, which is a special instance of Algorithm \ref{inertial_projective}, $\varphi_k(\cdot)$ is explicitly formulated by using points in the graphs of maximal monotone operators (see Eq. \eqref{eq:sep.F}).
\end{itemize}

\mgap

Next lemma plays the role of Fej\'er-monotonicity for Algorithm \ref{inertial_projective} and will be used in the proofs of
Theorems \ref{lm:conv.proj} and \ref{thm:second.conv}.

\begin{lemma}
 \label{lm:inv}
Consider the sequences evolved by \emph{Algorithm \ref{inertial_projective}} and let $\widetilde p^{k+1}$ be as
in \eqref{eq:001}. For an arbitrary $p\in \mathcal{S}$, define
\begin{align}
 \label{eq:def.hk}
  h_k=\norm{p^k-p}^2\qquad \forall k\geq -1.
\end{align}
Then the following hold:
\begin{itemize}
\item[\emph{(a)}] For all $k\geq 0$,
\begin{align*}
 h_{k+1}-h_k-\alpha_k(h_k-h_{k-1}) \leq \alpha_k(1+\alpha_k)\norm{p^k-p^{k-1}}^2 - s_{k+1},
\end{align*}
where
\begin{align}
  \label{eq:s.k}
  s_{k+1}:= \beta_k(2-\beta_k)\norm{\widehat p^{\,k}-\widetilde p^{\,k+1}}^2\qquad \forall k\geq 0.
 \end{align}
\item[\emph{(b)}] For all $k\geq 0$,
\begin{align}
  \label{eq:97}
   h_{k+1}-h_k-\alpha_k(h_k-h_{k-1}) \leq  \gamma_k \norm{p^k-p^{k-1}}^2
  - (2-\overline{\beta})\overline{\beta}^{\,-1}(1-\alpha_k)\norm{p^{k+1}-p^k}^2,
\end{align}
where
\begin{align}
 \label{eq:def.gamma}
 \gamma_k :=  2\left(1 - \overline{\beta}^{\,-1}\right)\alpha_k^2+2\overline{\beta}^{\,-1}\alpha_k\qquad \forall k\geq 0.
\end{align}
\end{itemize}
\end{lemma}
\begin{proof}
(a) We shall first prove that
\begin{align}
  \label{eq:1832}
 \norm{p^{k+1}-p}^2 + \beta_k(2-\beta_k)\norm{\widehat p^{\,k}-\widetilde p^{\,k+1}}^2 
\leq \norm{\widehat p^{\,k}-p}^2\qquad \forall p\in \mathcal{S},
\end{align}
where $\widetilde p^{\,k+1}$ is as in \eqref{eq:001}, i.e., it is the projection of $\widehat p^k$ onto
the semispace $\{p\in \HH\;|\;\varphi_k(p)\leq 0\}$.
To this end, note first that, for all $p\in \mathcal{S}$,
\begin{align}
\nonumber
 \norm{\widehat p^{\,k}-p}^2 - \norm{\widetilde p^{\,k+1}-p}^2 &=
 \norm{\widehat p^{\,k}-\widetilde p^{\,k+1}}^2+2\inner{\widehat p^{\,k}- \widetilde p^{\,k+1}}{\widetilde p^{\,k+1}-p}\\
  \label{eq:002}
       &\geq \norm{\widehat p^{\,k}-\widetilde p^{\,k+1}}^2
\end{align}
where we have used \eqref{eq:001} and the fact that
$\mathcal{S}\subset \{p\in \HH\,|\,\varphi_k(p)\leq 0\}$ (see Step 2 of Algorithm \ref{inertial_projective}) to obtain
the inequality $\inner{\widehat p^{\,k}- \widetilde p^{\,k+1}}{\widetilde p^{\,k+1}-p}\geq 0$.
Note now that \eqref{eq:333} is trivially equivalent to 
$p^{k+1}= (1-\beta_k)\widehat p^{\,k}+\beta_k \widetilde p^{\,k+1}$,
which in turn combined with the property \eqref{eq:ineq.s} yields
\begin{align*}
 \norm{p^{k+1}-p}^2= (1-\beta_k)\norm{\widehat p^{\,k}-p}^2 + \beta_k \norm{\widetilde p^{\,k+1}-p}^2
  - \beta_k(1-\beta_k)\norm{\widehat p^{\,k}-\widetilde p^{\,k+1}}^2
\end{align*}
or, equivalently, 
\begin{align}
 \label{eq:003}
\beta_k\left(\norm{\widehat p^{\,k}-p}^2 - \norm{\widetilde p^{\,k+1}-p}^2\right) = \norm{\widehat p^{\,k}-p}^2
- \beta_k(1-\beta_k)\norm{\widehat p^{\,k}-\widetilde p^{\,k+1}}^2 - \norm{p^{k+1}-p}^2.
\end{align}
The desired inequality \eqref{eq:1832} now follows by multiplying the inequality in \eqref{eq:002} by $\beta_k\geq 0$, 
by combining the resulting inequality with \eqref{eq:003} and by using some simple algebraic manipulations.

Now, from \eqref{eq:ext} we have
\begin{align}
 \label{eq:101}
  p^k-p=\frac{1}{1+\alpha_k}(\widehat p^{\,k}-p)+\frac{\alpha_k}{1+\alpha_k}(p^{k-1}-p)\;\;\mbox{and}\;\;
  \widehat p^{\,k}-p^{k-1}=(1+\alpha_k)(p^k-p^{k-1}).
\end{align}
Using \eqref{eq:ineq.s} and the first identity  in \eqref{eq:101} we obtain
\begin{align*}
  \norm{p^k-p}^2=\frac{1}{1+\alpha_k}\norm{\widehat p^{\,k}-p}^2+\frac{\alpha_k}{1+\alpha_k}\norm{p^{k-1}-p}^2-\frac{\alpha_k}{(1+\alpha_k)^2}\norm{\widehat p^{\,k}-p^{k-1}}^2,
\end{align*}
which combined with the second identity in \eqref{eq:101} and some algebraic manipulations gives
\begin{align}\label{eq:005}
  \norm{\widehat p^{\,k}-p}^2=(1+\alpha_k)\norm{p^k-p}^2-\alpha_k\norm{p^{k-1}-p}^2+\alpha_k(1+\alpha_k)\norm{p^k-p^{k-1}}^2.
\end{align}
Hence, (a) follows directly from \eqref{eq:1832}, \eqref{eq:005} and the definitions of $h_k$ and $s_{k+1}$ in
\eqref{eq:def.hk} and \eqref{eq:s.k}, respectively.

(b) Note that \eqref{eq:333} is also trivially equivalent to  $\widehat p^k - \widetilde p^{k+1} =  \beta_k^{-1}(\widehat p^k -p^{k+1})$, which in turn combined with the definition of $s_{k+1}$ in \eqref{eq:s.k} and the fact 
that $\beta_k\leq \overline{\beta}$ -- see Step 2 of Algorithm \ref{inertial_projective} -- yields
\begin{align}
 \label{eq:07}
 s_{k+1}=\beta_k(2-\beta_k) \norm{\widehat p^{\,k}-\widetilde p^{\,k+1}}^2
 =\left(2\beta_k^{-1}-1\right)\norm{\widehat p^{\,k}-p^{k+1}}^2
\geq \left(2\overline{\beta}^{\,-1}-1\right)\norm{\widehat p^{\,k}-p^{k+1}}^2.
\end{align}
Using \eqref{eq:ext}, the Cauchy-Schwarz inequality, the Young inequality ($2ab\leq a^2+b^2$ with $a=\norm{p^{k+1}-p^k}$ and $b=\norm{p^k-p^{k-1}}$) and some algebraic manipulations, we find
\begin{align}\label{eq:008}
  \norm{\widehat p^{\,k}-p^{k+1}}^2 = &\norm{p^{k+1}-p^{k}}^2+\alpha_k^2\norm{p^k-p^{k-1}}^2-2\alpha_k\langle p^{k+1}-p^k,p^k-p^{k-1}\rangle\nonumber \\
\geq&\norm{p^{k+1}-p^k}^2+\alpha_k^2\norm{p^k-p^{k-1}}^2-2\alpha_k\norm{p^{k+1}-p^k}\norm{p^k-p^{k-1}}\nonumber\\
\geq&\norm{p^{k+1}-p^k}^2+\alpha_k^2\norm{p^k-p^{k-1}}^2-\alpha_k\big(\norm{p^{k+1}-p^k}^2+\norm{p^k-p^{k-1}}^2\big)\nonumber\\
  =&(1-\alpha_k)\left(\norm{p^{k+1}-p^k}^2-\alpha_k \norm{p^k-p^{k-1}}^2\right).
 \end{align}
 From \eqref{eq:07} and \eqref{eq:008} we obtain
 \begin{align*}
  s_{k+1}\geq \left(2\overline{\beta}^{\,-1}-1\right)(1-\alpha_k)\left(\norm{p^{k+1}-p^k}^2-\alpha_k \norm{p^k-p^{k-1}}^2\right),
 \end{align*}
which in turn combined with the inequality in (a) and \eqref{eq:def.gamma}, and after some simple manipulations, gives exactly the desired inequality in (b).
\end{proof}

\mgap
Next is our first result on the (asymptotic) convergence of Algorithm \ref{inertial_projective}. The key assumption is the summability condition \eqref{eq:sum.p}, for which a sufficient condition (only depending on the parameters $\alpha_k$ and 
$\beta_k$) will be given in Theorem \ref{thm:second.conv} -- see conditions \eqref{eq:alpha_k}, \eqref{eq:beta(alpha)}
and Figure \ref{fig02}.

\begin{theorem}[First result on the convergence of Algorithm \ref{inertial_projective}]
 \label{lm:conv.proj}
 Let $\{p^k\}$, $\{\varphi_k\}$, $\{\widehat p^k\}$ and $\{\alpha_k\}$ be generated by \emph{Algorithm \ref{inertial_projective}} and
 assume that
 \begin{align}
 \label{eq:sum.p}
   \sum_{k=0}^\infty\,\alpha_k\norm{p^k-p^{k-1}}^2<\infty.
 \end{align}
Then the following hold:
\begin{itemize}
 \item[\emph{(a)}] $\{p^k\}$ and $\{\widehat p^k\}$ are bounded sequences.
 \item[\emph{(b)}] If every weak cluster point of $\{p^k\}$ belongs to $\mathcal{S}$, then $\{p^k\}$ converges weakly to some element in $\mathcal{S}$.
%
%
\item[\emph{(c)}] We have, 
\begin{align*}
 \dfrac{\max\{0,\varphi_k(\widehat p^k)\}}{\norm{\nabla \varphi_k}}\to 0.
\end{align*}
\end{itemize}
\end{theorem}
\begin{proof}
Defining $\delta_k=\alpha_k(1+\alpha_k)\norm{p^k-p^{k-1}}^2$ and using Lemma \ref{lm:inv}(a), we conclude that condition \eqref{eq:alv.att02} in Lemma \ref{lm:alv.att} below holds with $h_k$ and $s_{k+1}$ as in
\eqref{eq:def.hk} and \eqref{eq:s.k}, respectively. 
Hence, using the assumption \eqref{eq:sum.p}, Lemma \ref{lm:alv.att}(b) and \eqref{eq:def.hk}, we conclude that
$$\lim_{k\to \infty}\,\|p^k-p\|\;\; \mbox{exists for all}\;\; p\in \mathcal{S}.$$
This gives, in particular, that $\{p^k\}$ and $\{\widehat p^k\}$ are bounded (see \eqref{eq:ext})
and, after using Lemma \ref{lm:opial} below, that  $\{p^k\}$ converges weakly to some element in $\mathcal{S}$ whenever 
every weak cluster point of $\{p^k\}$ belongs to $\mathcal{S}$. So we have proved (a) and (b).

To prove (c), note first that from \eqref{eq:333} we have
\begin{align*}
 \dfrac{\max\{0,\varphi_k(\widehat p^k)\}}{\norm{\nabla \varphi_k}} = \norm{\widetilde p^{k+1}-\widehat p^k}.
\end{align*}
%
%
%
Hence, to conclude the proof of (c), it suffices to prove that $\norm{\widetilde p^{\,k+1}-\widehat p^{\,k}}\to 0$.
To this end, note that \eqref{eq:sum.p} combined with the definition of $\delta_k$ above,
the fact that $\alpha_k^2\leq \alpha_k$ 
 and Lemma \ref{lm:alv.att}(a) 
gives $\sum_{k=0}^\infty\,s_{k+1}<\infty$, with $s_{k+1}$ (for all $k\geq 0$) as in \eqref{eq:s.k}, and so
$s_{k+1}\to 0$. The desired result now follows form this fact, \eqref{eq:s.k} and the fact that 
$0<\underline{\beta}\leq \beta_k\leq \overline{\beta}<2$ (see Step 2 of Algorithm \ref{inertial_projective}).
\end{proof}

\mgap


\begin{theorem}[Second result on the convergence of Algorithm \ref{inertial_projective}]
\label{thm:second.conv}
  Let $\{p^k\}$ and $\{\alpha_k\}$ be generated by \emph{Algorithm \ref{inertial_projective}}.
  Assume that $\alpha\in [0,1)$, $\overline{\beta}\in (0,2)$ and $\{\alpha_k\}$ satisfy the following \emph{(}for some
  $\overline{\alpha}>0$\emph{)}:
  \begin{align}\label{eq:alpha_k}
      0\leq \alpha_k\leq \alpha_{k+1}\leq \alpha<\overline{\alpha}<1\qquad \forall k\geq 0
  \end{align}
 and
  \begin{align}\label{eq:beta(alpha)}
    \overline{\beta}=\overline{\beta}(\overline{\alpha}):=\dfrac{2(\overline{\alpha}-1)^2}
    {2(\overline{\alpha}-1)^2+3\overline{\alpha}-1}.
  \end{align}
  Then the following hold:
\begin{itemize}
\item[\emph{(a)}] We have
\begin{align}
   \label{sum(p_k)}
 \sum_{k=0}^{\infty}\,\norm{p^k-p^{k-1}}^2<\infty.
 \end{align}
\item[\emph{(b)}] Under the assumptions \eqref{eq:alpha_k} and \eqref{eq:beta(alpha)},
  if every weak cluster point of $\{p^k\}$ belongs to $\mathcal{S}$, then $\{p^k\}$ converges weakly to some element in $\mathcal{S}$.
\end{itemize}
 \end{theorem}
\begin{proof}
(a) Define, for all $k\geq 0$,
\begin{align}
 \label{eq:muka}
  \mu_k =h_k-\alpha_k h_{k-1}+\gamma_k\norm{p^k-p^{k-1}}^2
\end{align}
where $h_k$ is as in \eqref{eq:def.hk} (for some $p\in \mathcal{S}$) and $\gamma_k$ is as in \eqref{eq:def.gamma}.
Using the assumption \eqref{eq:alpha_k} and Lemma \ref{lm:inv}(b), we obtain, for all $k\geq 0$,
\begin{align}
 \label{eq:010}
 \mu_{k+1}-\mu_k
&\leq h_{k+1}-\alpha_{k}h_k+ \gamma_{k+1}\norm{p^{k+1}-p^k}^2-h_k+\alpha_kh_{k-1}-\gamma_k\norm{p^k-
 p^{k-1}}^2 \quad \text{[by \eqref{eq:alpha_k}]}\nonumber\\
&= h_{k+1}-h_k-\alpha_k(h_k-h_{k-1})+\gamma_{k+1}\norm{p^{k+1}-p^k}^2-\gamma_k\norm{p^k-p^{k-1}}^2 \nonumber \\
&\leq \left[-\left(2-\overline{\beta}\right)\overline{\beta}^{\,-1}(1-\alpha_k)+ \gamma_{k+1}\right]\norm{p^{k+1}-p^k}^2 \qquad\qquad\quad\text{[by Lemma \ref{lm:inv}(b)]}\nonumber\\
&\leq \left[-\left(2-\overline{\beta}\right)\overline{\beta}^{\,-1}(1-\alpha_{k+1})+\gamma_{k+1}\right]\norm{p^{k+1}-p^k}^2 \quad\qquad\qquad\qquad\qquad \text{[by \eqref{eq:alpha_k}]}\nonumber\\
&=-q(\alpha_{k+1})\norm{p^{k+1}-p^k}^2 \quad\qquad\qquad\qquad\qquad\qquad\qquad\qquad\quad
\text{[by \eqref{eq:def.gamma} and \eqref{eq:q(alpha)}]}
\end{align}
where
 \begin{align}\label{eq:q(alpha)}
 q(\nu):=2\left(\overline{\beta}^{\,-1}-1\right)\nu^2-\left(4\overline{\beta}^{\,-1}-1\right)\nu+2\overline{\beta}^{\,-1}-1,\quad \nu\in \mathbb{R}.
 \end{align}
Next we will show that $q(\alpha_{k+1})$ admits an uniform lower bound. To this end, note first that \eqref{eq:beta(alpha)} and Lemma \ref{lmm:inverse} below yield
\[
\overline{\alpha}=\dfrac{2(2-\overline{\beta})}{4-\overline{\beta}+\sqrt{16\overline{\beta}-7\overline{\beta}^2}},
\]
which in turn combined with Lemma \ref{lm:quadratic} below implies that $q(\overline{\alpha})=0$ and $q(\cdot)$ is decreasing in $[0,\overline{\alpha}]$. Thus, in view of \eqref{eq:alpha_k},
 we obtain
\[
q(\alpha_{k+1})\geq q(\alpha)>q(\overline{\alpha})=0
\]
and so, in view of \eqref{eq:010}, it follows that
\begin{align}
  \label{eq:2351}
 \norm{p^{k+1}-p^k}^2\leq \dfrac{1}{q(\alpha)}(\mu_k - \mu_{k+1})\qquad \forall k\geq 0.
\end{align}
Hence, for all $k\geq 0$,
\begin{align}
 \label{eq:2348}
\sum_{j=0}^k\,\norm{p^{j+1}-p^j}^2&\leq \dfrac{1}{q(\alpha)}(\mu_0 - \mu_{k+1})\nonumber\\
      &\leq \dfrac{1}{q(\alpha)}(\mu_0 + \alpha h_k)
\end{align}
where in the second inequality above we also used the fact that $\mu_{k+1}\geq -\alpha h_k$ (in view of \eqref{eq:muka}
and \eqref{eq:alpha_k}).
Therefore, to finish the proof of (a) it is enough to find an upper bound on $h_k$ and use \eqref{eq:2348}.
To this end, note that from \eqref{eq:2351} and \eqref{eq:alpha_k} we have, for all $k\geq -1$,
\begin{align*}
\mu_0\geq  \mu_1\geq \ldots\geq \mu_{k+1}&=h_{k+1}-\alpha_{k+1}h_{k}+
\gamma_{k+1}\norm{p^{k+1}-p^{k}}^2\\
   &\geq h_{k+1}-\alpha  h_{k}
\end{align*}
and so, for all $k\geq -1$,
\begin{align*}
 h_{k+1}&\leq \alpha^{k+1}h_0 + \left(\sum_{i=0}^k\,\alpha^i\right)\mu_0\\
     &\leq h_0 + \dfrac{\mu_0}{1-\alpha}
\end{align*}
where in the second inequality we also used the fact -- from \eqref{eq:muka} -- that $\mu_0=(1-\alpha_0)h_0\geq 0$.
(b) The result follows trivially from (a), the fact that $\alpha_k\leq 1$ for all $k\geq 0$ and Theorem \ref{lm:conv.proj}(b). 
\end{proof}

 \begin{figure}
    \centering
        \begin{tikzpicture}[scale=2] \centering
    \draw[->,line width = 0.50mm] (-0.2,0) -- (1.4,0) node[right] { $\overline{\alpha}$};
    \draw[->,line width = 0.50mm] (0,-0.2) -- (0,2.4) node[above] {$\overline{\beta}(\overline{\alpha})$};
    \draw[domain=0:1,smooth,variable=\x,red,line width = 0.50mm,scale=1] plot ({\x},{(2*(\x -1)^2)/((2*(\x -1)^2)+3*\x -1)});
    \draw[red] (0,2) circle (0.35mm);
    \draw[red] (1,0) circle (0.35mm);
    \draw[dashed] (0.3333,0) -- (0.3333,1);
     \draw[dashed] (0,1) -- (0.3333,1);
        \node[below left,black] at (0,0) {0};
        \node[below right ,black] at (1,0) {1};
        \node[above left,black] at (0,2) {2};
        \node[below ,black] at (0.3333,0) {$\frac{1}{3}$};
        \node[left,black] at (0,1) {1};
    \end{tikzpicture}
    \caption{{The relaxation parameter upper bound
    $\overline{\beta}(\overline{\alpha})$ from \eqref{eq:beta(alpha)} 
    as a function of inertial step upper bound $\overline{\alpha}>0$
    of \eqref{eq:alpha_k}. Note that $\overline \beta(1/3)=1$, while
    $\overline{\beta}(\overline{\alpha})>1$ whenever $\overline \alpha<1/3$.}}
    \label{fig02}
    \end{figure}
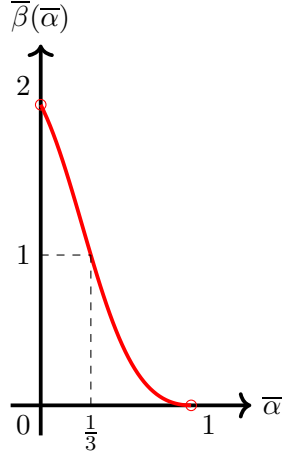

\mgap
\mgap
\noindent
{\bf Remarks}.
\begin{itemize}
\item[\mbox{(i)}] The proofs of Theorems \ref{lm:conv.proj} and \ref{thm:second.conv} have followed the same outline of the
proofs of Theorems 2.4 and 2.5 in \cite{alv.eck.ger.mel-preprint19}. On the other hand, we emphasize that
Algorithm \ref{inertial_projective} proposed in this work is more general that Algorithm 1 from \cite{alv.eck.ger.mel-preprint19}, since
the latter has been designed to solve inclusions with monotone operators.
\item[\mbox{(ii)}] We deduce from conditions \eqref{eq:alpha_k} and \eqref{eq:beta(alpha)} that overrelaxation effects
can be achieved in Algorithm \ref{inertial_projective} at the price of choosing the inertial parameter upper bound
$\overline{\alpha}$ strictly smaller than $1/3$ (see Figure \ref{fig02}). We also emphasize that the interplay between inertial and 
relaxation effects has also been investigated, e.g., in 
\cite{att.cab-con.jde18,att.pey-con.mp19,bot.cse.hen-ine.amc15,com.gla-qua.sjo17}.
\end{itemize}

%
\section{A relative-error inertial-relaxed inexact  projective splitting algorithm}
 \label{sec:ps}
In this section, we propose and study the asymptotic convergence of a relative-error inertial-relaxed inexact projective splitting algorithm 
(Algorithm \ref{alg:iPSM}). The main convergence results are stated in Theorems \ref{th:conv01} and \ref{th:conv02}.

We start by considering the monotone inclusion problem \eqref{eq:Pmip} (or, equivalently, \eqref{eq:Pmip02}), i.e., 
 the problem of finding $z\in \HH_0$ such that
\begin{align}
 \label{eq:Pmip1}
 0\in \sum_{i=1}^{n}G_i^*T_iG_i(z)
\end{align}
where $n\geq 2$ and Assumptions (A1)--(A3) of Section \ref{sec:int} are assumed to hold. 

Consider the extend solution set  (or generalized Kuhn-Tucker set) as in \eqref{eq:KTSi} for the problem \eqref{eq:Pmip1}, i.e.:
\begin{align}
\label{eq:KTS}
  \mathcal{S}:=\left\{(z,w_1,\ldots,w_{n-1})\in \boldsymbol{\HH}\;\;|\;\; w_i\in T_i(G_iz),\, i=1,\ldots,n-1,\, -\sum_{i=1}^{n-1}G^*_iw_i\in T_n(z) \right\}.
\end{align}
As we pointed out early, $z\in \HH_0$  is a solution of \eqref{eq:Pmip1} if and only if there exist $w_i\in \HH_i$ ($i=1,\ldots,n-1$)  such that $(z,w_1,\ldots, w_{n-1})\in \mathcal{S}$. 
We deduce from Assumption (A3) above that $\mathcal{S}$ is nonempty. Moreover, it follows form 
\cite[Lemma 3]{Johnst.Eckst.2018} that
$\mathcal{S}$ is closed and convex in $\boldsymbol{\HH}$ (endowed with inner product and norm as in \eqref{eq:def.inner}). 
As a consequence, one can apply the framework (Algorithm \ref{inertial_projective}) of Section \ref{sec:spm} for $\mathcal{S}$ as in \eqref{eq:KTS} and the Hilbert space 
$\boldsymbol{\HH}$ with the inner product and norm as in \eqref{eq:def.inner}. The resulting scheme is Algorithm \ref{alg:iPSM}, which, in particular, will be shown in Proposition \ref{prop:Q1} to be a special instance of Algorithm \ref{inertial_projective}.

Since Step 2 of Algorithm \ref{inertial_projective} demands the construction of an (nonconstant) affine function $\varphi_k$ such that
$\varphi_k(p)\leq 0$ for all $p\in \mathcal{S}$, next we discuss the construction of such $\varphi_k$ satisfying the latter inequality for
$\mathcal{S}$ as in \eqref{eq:KTS}.

Motivated by \eqref{eq:sep.i} and \eqref{eq:wni}, for  $y_i^k\in T_i(x_i^k)$ ($i=1,\dots, n$), we define $\varphi_k:\boldsymbol{\HH}\to \R$ by
\begin{align}
    \label{eq:sep.F02}
  \varphi_k(\underbrace{z,w_1,\ldots,w_{n-1}}_{p})=\sum_{i=1}^{n-1}\langle G_iz-x_i^k,y_i^k-w_i\rangle+
   \langle G_n z-x_n^k,y_n^k + \sum_{i=1}^{n-1}G_i^*w_i\rangle.
 \end{align}
We shall also use the fact, from \eqref{eq:sep.F02} and \eqref{eq:wni}, that
\begin{align}
 \label{eq:sep.F03}
 \varphi_k(p) = \sum_{i=1}^{n}\langle G_iz-x_i^k,y_i^k-w_i\rangle.
\end{align}
Note that the construction above depends on the computation of pairs $(x_i^k,y_i^k)$ in the graph of $T_i$, for each $i=1,\dots, n$, which
can be computed by inexact evaluation (with relative-error tolerance) of the resolvent $J_{T_i}=(T_i+I)^{-1}$ of $T_i$ (see Step 2 of Algorithm \ref{alg:iPSM}).

\mgap

Next lemma presents some properties of $\varphi_k$ which will be useful in this paper.

\begin{lemma}\emph{(\cite[Lemma 4]{Johnst.Eckst.2018})}
 \label{lmm:grad.sep}
 Let $\varphi_k(\cdot)$ and $\mathcal{S}$ be as in \eqref{eq:sep.F02} and \eqref{eq:KTS}, respectively.
 The following hold:
   \begin{enumerate}
     \item [\emph{(a)}] $\varphi_k$  is affine on $\boldsymbol{\HH}$.
     \item [\emph{(b)}]  $\varphi_k(p)\leq 0$ for all $p\in \mathcal{S}$.
     \item [\emph{(c)}]  The gradient of $\varphi_k$  with respect to the inner product $\inner{\cdot}{\cdot}_{\gamma}$ as in \eqref{eq:def.inner} is
     \begin{align}\label{eq:grad.phi_k}
       \nabla\varphi_k=
              \left(\dfrac{1}{\gamma}\left(\sum_{i=1}^{n-1}G_i^*y_i^k+y_n^k\right),x_1^k-G_1x_n^k,
              \ldots,x_{n-1}^k-G_{n-1}x_{n}^k\right).
     \end{align}
     \item [\emph{(d)}] If $\nabla \varphi_k=0$, then $(x_n^k,y_1^k,\ldots,y_{n-1}^k)\in \mathcal{S}$.
   \end{enumerate}
 \end{lemma}

\mgap

As a direct consequence of Lemma \ref{lmm:grad.sep}(c) and \eqref{eq:def.inner}, we have
\begin{align}
 \label{eq:norm.grad.pi_k}
  \norm{\nabla\varphi_k}_{\gamma}^2=
  \gamma^{-1}\left\|\sum_{i=1}^{n-1}G_i^*y_i^k+y_n^k\right\|^2+\sum_{i=1}^{n-1}\norm{x_i^k-G_ix_{n}^k}^2.
\end{align}

 \mgap

Next we present the main algorithm of this paper. As we mentioned before, it consists of a relative-error inertial-relaxed inexact projective splitting method for solving \eqref{eq:Pmip1}.

\mgap

\mgap
\noindent
\fbox{
\addtolength{\linewidth}{-2\fboxsep}%
\addtolength{\linewidth}{-2\fboxrule}%
\begin{minipage}{\linewidth}
\begin{algorithm}
\label{alg:iPSM}
{\bf  A relative-error inertial-relaxed inexact projective splitting algorithm for solving \eqref{eq:Pmip1}}
\end{algorithm}
\begin{itemize}
\item[{\bf(0)}] Let $(z^{-1},w_1^{-1},\ldots,w_{n-1}^{-1})=(z^0,w_1^{0},\ldots, w_{n-1}^{0}) \in \boldsymbol{\HH}$, 
$0\leq \alpha,\sigma<1$, $0<\underline{\beta}\leq \overline{\beta}< 2$ and $\gamma >0$ be given; let $k\leftarrow0$.
\item [{\bf(1)}]  Choose $\alpha_k\in [0,\alpha]$ and let
\begin{align}
\label{eq:ext.022}
 &\widehat z^{\,k}=z^k+\alpha_k(z^k-z^{k-1}),\\
 \label{eq:ext.02}
 &\widehat w_i^{\,k}= w_i^{k}+\alpha_k(w_i^k-w_i^{k-1}),\quad i=1,\ldots,n-1,\\
 \label{eq:ext.023}
 &\widehat w_n^{\,k}=-\sum_{i=1}^{n-1}G_i^*\widehat w_i^{\,k}.
\end{align}
\item [{\bf(2)}]  Choose scalars $\rho_i^k>0$ and compute $(x_i^k,y_i^k)$ ($i=1,\ldots,n$) satisfying
\begin{align}
\left\{
       \begin{array}{ll}
			 \label{eq:proxT_i}
	       y_i^k\in T_i(x_i^k),\quad \rho_i^k y_i^k + x_i^k = G_i\widehat z^{\,k}+\rho_i^k \widehat w_i^{\,k}+e_i^k,\\[5mm]
                 \norm{e_i^k}^2\leq \sigma^2\left(\norm{G_i\widehat z^k-x_i^k}^2+\norm{\rho_i^k(\widehat w_i^k-y_i^k)}^2\right).
        \end{array}
        \right.
\end{align}
%
%
%
\item [{\bf(3)}]
If $x_i^k=G_i x_n^k$ ($i=1,\ldots,n-1$) and $\displaystyle \sum_{i=1}^{n-1}G_i^*y_i^k +y_n^k=0$, then STOP.
Otherwise, define
\begin{align}
  \label{eq:sep.F}
 &\varphi_k (\widehat p^{\,k}) =\sum_{i=1}^{n-1}\langle G_i\widehat z^k-x_i^k,y_i^k - \widehat w_i^k\rangle+
  \langle G_n \widehat z^k-x_n^k,y_n^k + \sum_{i=1}^{n-1}G_i^*\widehat w^k_i\rangle,\\
\label{eq:theta_k}
&\theta_k =\dfrac{\max\{0, \varphi_k(\widehat p^{\,k}) \}}{\gamma^{-1}\norm{\sum_{i=1}^{n-1}G_i^*y_i^k+y_n^k}^2+  \sum_{i=1}^{n-1}\norm{x_i^k-G_i x_n^k}^2}.
\end{align}
\item[{\bf(4)}] Choose some relaxation parameter $\beta_k\in [\underline{\beta},\overline{\beta}]$ and define
\begin{align}
\label{eq:ext.proj02}
&z^{k+1}=\widehat z^{\,k}-\gamma^{-1}\beta_k\theta_k\left(\sum_{i=1}^{n-1}G_i^*y_i^k+y_n^k \right),\\
&w_i^{k+1}=\widehat w_i^{\,k}-\beta_k\theta_k\left(x_i^k-G_i x_n^k\right), \quad i=1,\ldots,n-1.\label{eq:ext.proj002}
\end{align}
\item[{\bf(5)}] Let $k\leftarrow k+1$ and go to step 1.
\end{itemize}
\end{minipage}
} 

\mgap
\mgap
\mgap
\mgap

\noindent
{\bf Remarks.}

\begin{itemize}
 \item[\mbox{(i)}] Similarly to Algorithm \ref{inertial_projective} of Section \ref{sec:spm}, Algorithm \ref{alg:iPSM} also promotes
 inertial and relaxation effects, controlled by the parameters $\alpha_k$ and $\beta_k$, respectively. The inertial (extrapolation) step
 is performed in \eqref{eq:ext.022} and \eqref{eq:ext.02}, while the relaxed projective step is given in \eqref{eq:ext.proj02} and
 \eqref{eq:ext.proj002} (in the context of Algorithm \ref{inertial_projective}, see Figure \ref{fig:arrows} of Section \ref{sec:spm}). Conditions on the choice of the upper bounds $\alpha$ and $\overline{\beta}$, as well as on the sequence of extrapolation parameters
 $\{\alpha_k\}$, to guarantee the convergence of Algorithm \ref{alg:iPSM} will be given in Theorem \ref{th:conv02}. 
\item[\mbox{(ii)}] Direct substitution of \eqref{eq:ext.02} into \eqref{eq:ext.023} gives that, similarly to
$\widehat w_i^k$ for $i=1,\dots, n-1$, $\widehat w_n^k$ also satisfies
\begin{align}
  \widehat w_n^k = w_n^k + \alpha_k(w_n^k - w_n^{k-1}),
\end{align}
where
 \begin{align}
\label{eq:w_nk}
 w_n^k := -\sum_{i=1}^{n-1}G_i^* w_i^k,\qquad \forall k\geq 0.
\end{align}
\item[\mbox{(iii)}] The computation of $(x_i^k,y_i^k)$ in \eqref{eq:proxT_i} can be performed inexactly within a relative-error tolerance 
controlled by the parameter $\sigma\in [0,1)$. In practice, the error condition in \eqref{eq:proxT_i} is used as a stopping-criterion for
some computational procedure (e.g., the conjugate gradient algorithm) applied to (inexactly) solving the related inclusion (for $i=1,\dots, n$)
\begin{align*}
 0\in \rho_i^k T_i(x) + x- (G_i\widehat z^k +\rho_i^k \widehat w_i^k)
\end{align*}
until the error-condition in \eqref{eq:proxT_i} is satisfied for the first time. Note also that $(x_i^k,y_i^k)$ is given explicitly by
$x_i^k=J_{\rho_i^k T_i}(G_i\widehat z^k +\rho_i^k \widehat w_i^k)$ and 
$y_i^k = \frac{G_i\widehat z^k - x_i^k}{\rho_i^k}+\widehat w_i^k$ whenever the resolvent $J_{\rho_i^k T_i} = (\rho_i^k T_i +I)^{-1}$ of $T_i$ is assumed to be easily computed and $\sigma=0$ in \eqref{eq:proxT_i}.
For the particular case of the minimization problem \eqref{eq:opt}, the computation of $(x_i^k,y_i^k)$ reduces to the (inexact) computation of the proximity operator $\mbox{prox}_{\rho_i^k f_i}$, i.e., in this case
\begin{align}
  \label{eq:prox_fi}
 x_i^k \approx \mbox{arg}\min_{z\in \HH_0}\,\left\{f_i(z)+\dfrac{1}{2\rho_i^k}\norm{z-(G_i\widehat z^k+\rho_i^k\widehat w_i^k)}^2\right\}.
\end{align}
See also Section \ref{sec:ne} for an additional discussion in the context of LASSO problems.
%
%
%
\item[\mbox{(iv)}] It follows from Lemma \ref{lmm:grad.sep}, items (c) and (d), that $(x_n^k,y_1^k,\dots, y_{n-1}^k)$ belongs
to the extended solution set $\mathcal{S}$ whenever Algorithm \ref{alg:iPSM} stops at Step 3. In particular, in this case, $x_n^k$ is a solution of
\eqref{eq:Pmip1}. 

\emph{
\begin{center}
Motivated by Remark \emph{(iv)} above, from now one in this paper we assume that \emph{Algorithm \ref{alg:iPSM}} generates infinite sequences, i.e.,
we assume
that it never stops at \emph{Step 3}.
\end{center}
}

\mgap

\item[\mbox{(v)}] We also emphasize that if $\alpha_k\equiv 0$ in Algorithm \ref{alg:iPSM}, then it reduces to the projective splitting
algorithm (or some of its variants) originated in \cite{eck.sva-gen.sjco09} and later developed in different directions
(see, e.g., 
\cite{
combettes2015BestAppr,
alot.combet.shah.2014,
combettes2020warped,
combettes2022saddle,
combettes2018asynchronous,
combettes2016Solving,
eckstein2019rates,
Johnst.Eckst.2019,
Johnst.Eckst.2018,
jonhst.eckst.siam2019}). 
The advantages and flexibility of projective splitting algorithms (beyond inertial effects) when compared to other proximal-splitting strategies are also extensively discussed in the latter references.
\end{itemize}

\mgap

Next we show that Algorithm \ref{alg:iPSM} (under the assumption that it never stops at Step 3; see Remark (iv) above) is a special instance of Algorithm \ref{inertial_projective} for finding a point in $\mathcal{S}$
as in \eqref{eq:KTS} in the Hilbert space $\boldsymbol{\HH}$ endowed with the inner product and norm as in \eqref{eq:def.inner}.

\mgap

 %
 %
%
%
%
%

\begin{proposition}
\label{prop:Q1}
 Assume that  \emph{Algorithm \ref{alg:iPSM}} does not stop at \emph{Step 3}, let
$\{z^k\}$, $\{w_1^k\}, \dots, \{w_{n-1}^k\}$  be generated by \emph{Algorithm \ref{alg:iPSM}}, let $\{\varphi_k\}$ be as
in \eqref{eq:sep.F02} and  define
\begin{align}
 \label{eq:def.pk2}
 p^k = (z^k,w_1^k,\dots, w_{n-1}^k)\qquad \forall k\geq -1.
\end{align}
Then the following hold:
\begin{itemize}
\item[\emph{(a)}] For all $k\geq 0$,
\[
  \nabla \varphi_k \neq 0\;\;\mbox{and}\;\; \varphi_k(p)\leq 0\qquad \forall p\in \mathcal{S},
\]
where $\mathcal{S}$ is as in \eqref{eq:KTS}.
\item[\emph{(b)}] For all $k\geq 0$,
\begin{align}
 \label{eq:prop:Q101}
  p^{k+1}=\widehat p^{\,k} - \dfrac{\beta_k \max\{0,\varphi_k(\widehat p^{\,k})\}}{\norm{\nabla \varphi_k}_\gamma^2}\nabla \varphi_k\;\;\emph{and}\;\; \widehat p^k = (\widehat z^k,\widehat w_i^k,\dots, \widehat w_{n-1}^k),
 \end{align}
where $\widehat p^k$ is as in \eqref{eq:ext}. 
\end{itemize}
 As a consequence of \emph{(a)} and \emph{(b)} above, it follows that \emph{Algorithm \ref{alg:iPSM}} is a special instance of \emph{Algorithm \ref{inertial_projective}} for finding a point in the
extended solution set $\mathcal{S}$ as in \eqref{eq:KTS}.
 \end{proposition}
\begin{proof}
(a) The fact that $\nabla \varphi_k\neq 0$ follows from the assumption that Algorithm \ref{alg:iPSM} does not
stop at Step 3 and Lemma \ref{lmm:grad.sep}(c). 
Using now Lemma \ref{lmm:grad.sep}(b) and the inclusions in \eqref{eq:proxT_i}, we conclude that
$\varphi_k(p)\leq 0$ for all $p\in \mathcal{S}$.  

(b) The second identity in \eqref{eq:prop:Q101} follows from \eqref{eq:ext}, \eqref{eq:def.pk2}, \eqref{eq:ext.022} and \eqref{eq:ext.02}.
On the other hand, the first identity in \eqref{eq:prop:Q101} is a direct consequence of \eqref{eq:theta_k}--\eqref{eq:ext.proj002},
\eqref{eq:grad.phi_k}, \eqref{eq:norm.grad.pi_k} and the second identity in \eqref{eq:prop:Q101}. 

Finally, the last statement of the proposition is a consequence of items (a) and (b) as well as of Algorithm \ref{inertial_projective}'s definition.
\end{proof}

\mgap

Since Algorithm \ref{alg:iPSM} is a special instance of Algorithm \ref{inertial_projective} of Section \ref{sec:spm}, it follows from 
Theorems \ref{lm:conv.proj}(b) and \ref{thm:second.conv}(b), under
the assumptions \eqref{eq:sum.p} and \eqref{eq:alpha_k}--\eqref{eq:beta(alpha)}, respectively, that to prove the convergence of Algorithm \ref{alg:iPSM} it suffices to check that every weak cluster point of Algorithm \ref{alg:iPSM} belongs to $\mathcal{S}$ as in \eqref{eq:KTS}. This will be done in Proposition \ref{pr:varios}(e), but before we need the lemma below.

\begin{lemma}
 \label{lmm:seg01_new}
Consider the sequences evolved by \emph{Algorithm \ref{alg:iPSM}}, let 
$\widehat p^k = (\widehat z^k,\widehat w_i^k,\dots, \widehat w_{n-1}^k)$ and let $\widehat w_n^k$ be as in \eqref{eq:ext.023}.
Assume that, for $i=1,\dots, n$,
\begin{align}
  \label{eq:assum.rho}
  0<\underline{\rho}\leq \rho_i^k\leq \overline{\rho}<\infty\qquad \forall k\geq 0.
\end{align}
Then the following hold:
\begin{itemize}
\item[\emph{(a)}] For all $k\geq 0$,
\begin{align}\label{eq:max.phi_k}
  \varphi_k(\widehat p^k)\geq
  \dfrac{(1-\sigma^2)\min\left\{\overline{\rho}^{\,-1},\underline{\rho}\right\}}{2}
  \sum_{i=1}^n\,
  \left(\norm{G_i\widehat z^k-x_i^k}^2+\norm{\widehat w_i^k-y_i^k}^2\right)\geq 0.
  \end{align}
 \item[\emph{(b)}] There exists a constant $c>0$ such that, for all $k\geq 0$,
 \begin{align}
   \label{eq:smale}
   \dfrac{\varphi_k(\widehat p^k)^2}{c\norm{\nabla \varphi_k}_\gamma^2}\geq 
   \varphi_k(\widehat p^k)\geq c\,\norm{\nabla \varphi_k}^2_\gamma.
 \end{align}
\end{itemize}
\end{lemma}
\begin{proof}
(a) Using the identity $\inner{a}{b}=(1/2)\left(\norm{a+b}^2-\norm{a}^2-\norm{b}^2\right)$ 
with $a=x_i^k-G_i\widehat z^k$ and $b=\rho_i^k(y_i^k-\widehat w_i^k)$, and some algebraic manipulations, 
we obtain, for $i=1,\dots, n$,
\begin{align*}
 \inner{x_i^k-G_i\widehat z^k}{\rho_i^k(y_i^k-\widehat w_i^k)} &
      = \dfrac{1}{2}\left(\norm{\underbrace{x_i^k-G_i\widehat z^k+\rho_i^k(y_i^k-\widehat w_i^k)}_{= e^k_i}}^2
       -\norm{G_i\widehat z^k-x_i^k}^2-\norm{\rho_i^k(\widehat w_i^k-y_i^k)}^2\right)\\
     & \leq \dfrac{-(1-\sigma^2)}{2}\left(\norm{G_i\widehat z^k-x_i^k}^2+\norm{\rho_i^k(\widehat w_i^k-y_i^k)}^2\right),
\end{align*}
where we also used the error condition in \eqref{eq:proxT_i}. Note now that the desired result follows by dividing the latter
inequality by $-\rho_i^k$ and by using \eqref{eq:sep.F03} and assumption \eqref{eq:assum.rho}.

(b) First note that using the property \eqref{eq:ineq.sum2}, \eqref{eq:ext.023} and the assumption that $G_n=I$, we obtain
\begin{align}
  \label{eq:0946}
 \left\|\sum_{i=1}^{n}G_i^*y_i^k\right\|^2 = \left\|\sum_{i=1}^{n}G_i^*(\widehat w_i^k-y_i^k)\right\|^2
 \leq n\left(\max_{i=1,\dots, n}\,\norm{G_i^*}^2\right)\,\sum_{i=1}^n\norm{\widehat w_i^k-y_i^k}^2.
\end{align}
On the other hand, using the inequality $\norm{a+b}^2\leq 2(\norm{a}^2+\norm{b}^2)$, (again) the fact that $G_n=I$
and some algebraic manipulations, we find
\begin{align}
\nonumber
 \sum_{i=1}^{n-1}\norm{x_i^k-G_i x_n^k}^2 &= \sum_{i=1}^{n-1}\norm{x_i^k - G_i\widehat z^k + G_i(\widehat z^k-x_n^k)}^2\\
 \nonumber
     &\leq 2\sum_{i=1}^{n-1}\,\left(\norm{G_i\widehat z^k-x_i^k}^2+\norm{G_i(\widehat z^k-x_n^k)}^2\right)\\
     \nonumber
     &\leq 2\left(\sum_{i=1}^{n-1}\,\norm{G_i\widehat z^k-x_i^k}^2+(n-1)\max_{i=1,\dots, n-1}\{\norm{G_i}^2\}
      \norm{\widehat z^k-x_n^k}^2\right)\\
      \label{eq:avila}
      &\leq 2 \max\left\{1, (n-1)\max_{i=1,\dots, n-1}\{\norm{G_i}^2\}\right\}\,\sum_{i=1}^n\,\norm{G_i\widehat z^k-x_i^k}^2.
\end{align}
We know from \eqref{eq:norm.grad.pi_k} (and the fact that $G_n=I$) that
\begin{align*}
\norm{\nabla\varphi_k}_{\gamma}^2&=
  \gamma^{-1}\left\|\sum_{i=1}^{n}G_i^*y_i^k\right\|^2+\sum_{i=1}^{n-1}\norm{x_i^k-G_ix_{n}^k}^2,
\end{align*}
which combined with \eqref{eq:0946}, \eqref{eq:avila} and \eqref{eq:max.phi_k} yields the second inequality in \eqref{eq:smale}, for some constant $c>0$. To finish the proof, note that the first inequality in \eqref{eq:smale} is a direct consequence of the second one.
\end{proof}

\mgap
\mgap

\begin{proposition}
 \label{pr:varios}
 Consider the sequences evolved by \emph{Algorithm \ref{alg:iPSM}} and let $\{w_n^k\}$ and $\{p^k\}$ be as in \eqref{eq:w_nk}
 and \eqref{eq:def.pk2}, respectively. 
 %
%
 Assume that
 \begin{align}
  \label{eq:901}
   \sum_{k=0}^\infty\,\alpha_k\norm{p^k-p^{k-1}}_\gamma^2<\infty
 \end{align}
and, for $i=1,\dots, n$,
 \begin{align}
  \label{eq:assum.rho2}
  0<\underline{\rho}\leq \rho_i^k\leq \overline{\rho}<\infty\qquad \forall k\geq 0.
\end{align}
 Then,
 \begin{itemize}
 \item[\emph{(a)}] We have, $\varphi_k(\widehat p^k)\to 0$ and $\norm{\nabla \varphi_k}_\gamma \to 0$.
%
%
\item[\emph{(b)}] We have, $\sum_{i=1}^n\,G_i^* y_i^k\to 0$ and $x_i^k-G_i x_n^k\to 0$ \emph{(}$i=1,\dots, n-1$\emph{)}.

\item[\emph{(c)}] For each $i=1,\dots, n$, we have 
 $\norm{G_i\widehat z^k-x_i^k}\to 0$ and $\norm{\widehat w_i^k-y_i^k}\to 0$.
%
%
 \item[\emph{(d)}]  For each $i=1,\dots, n$, we have $\norm{G_i z^k-x_i^k}\to 0$ and $\norm{w_i^k-y_i^k}\to 0$.
 \item[\emph{(e)}] Every weak cluster point of $\{p^k\}$ belongs to $\mathcal{S}$, where $\mathcal{S}$ is as in \eqref{eq:KTS}.
 \end{itemize}
\end{proposition}
\begin{proof}
(a) Using the last statement in Proposition \ref{prop:Q1}, Theorem \ref{lm:conv.proj}(c) and the fact from \eqref{eq:max.phi_k} that
$\varphi_k(\widehat p^k)\geq 0$, we obtain
\begin{align*}
 \dfrac{\varphi_k(\widehat p^k)}{\norm{\nabla \varphi_k}_\gamma}\to 0,
\end{align*}
which after taking limit in \eqref{eq:smale} gives the desired result in item (a).

(b) This follows from the second limit in item (a) combined with \eqref{eq:norm.grad.pi_k} (and the fact that $G_n=I$).

(c) This follows from the first limit in item (a) and \eqref{eq:max.phi_k}.

(d) Using the triangle inequality, the identity \eqref{eq:ext.022}, \eqref{eq:def.pk2} and \eqref{eq:def.inner},
we find
\begin{align}
\nonumber
 \norm{G_iz^k-x_i^k}&\leq \norm{z^k-\widehat z^{\,k}}\norm{G_i}+\norm{G_i \widehat z^{\,k}-x_i^k}\\
 \nonumber
  & = \alpha_k\norm{z^k-z^{k-1}}\norm{G_i}+\norm{G_i\widehat z^{\,k}-x_i^k}\\
  \label{eq:905}
  &\leq \sqrt{\gamma^{-1}}\sqrt{\alpha_k}\,\norm{p^k-p^{k-1}}_\gamma\norm{G_i}+\norm{G_i\widehat z^{\,k}-x_i^k},
  \quad i=1,\dots, n,
\end{align}
where we also used the fact that $\alpha_k\leq \sqrt{\alpha_k}$ (because $0\leq \alpha_k<1$).
Using a similar reasoning, we also find
\begin{align}
\label{eq:TM01}
 \norm{y_i^k-w_i^k} \leq \sqrt{\alpha_k}\,\norm{p^k-p^{k-1}}_\gamma+\norm{y_i^k-\widehat w_i^k},\qquad i=1,\dots, n-1.
 \end{align}
Note also that, using \eqref{eq:ext.023}, \eqref{eq:ext.02}, \eqref{eq:w_nk}, the property \eqref{eq:ineq.sum2}, the fact that $\alpha_k^2\leq \alpha_k$ and \eqref{eq:def.inner}, we obtain
\begin{align}
 \nonumber
 \dfrac{1}{2}\norm{y_n^k-w_n^k}^2 &\leq \norm{\widehat w_n^k - w_n^k}^2+\norm{y_n^k- \widehat w_n^k}^2\\
 \nonumber
        & \leq (n-1)\max_{i=1,\dots, n-1}\{\norm{G_i^*}^2\}
        \left(\sum_{i=1}^{n-1}\,\norm{\widehat w_i^k-w_i^k}^2\right)+\norm{y_n^k- \widehat w_n^k}^2\\
        \nonumber
        & = (n-1)\max_{i=1,\dots, n-1}\{\norm{G_i^*}^2\}
        \left(\sum_{i=1}^{n-1}\,\alpha_k^2\norm{w_i^{k-1}-w_i^k}^2\right)+\norm{y_n^k- \widehat w_n^k}^2\\
        \nonumber
        & \leq (n-1)\max_{i=1,\dots, n-1}\{\norm{G_i^*}^2\}
        \left(\sum_{i=1}^{n-1}\,\alpha_k\norm{w_i^{k-1}-w_i^k}^2\right)+\norm{y_n^k- \widehat w_n^k}^2\\
        \label{eq:904}
        & \leq (n-1)\max_{i=1,\dots, n-1}\{\norm{G_i^*}^2\}
        \alpha_k\norm{p^k-p^{k-1}}^2_\gamma+\norm{y_n^k- \widehat w_n^k}^2.
\end{align}
To finish the proof of (d), combine \eqref{eq:905}--\eqref{eq:904} with item (c) and assumption \eqref{eq:901} 
(which, in particular, implies that $\alpha_k\norm{p^k-p^{k-1}}_\gamma^2\to 0$).

(e) Let $p^\infty:=(z^\infty, w_1^\infty, \cdots, w_{n-1}^\infty)\in \boldsymbol{\HH}$ be a weak cluster point of $\{p^k\}$ (by Proposition
\ref{prop:Q1} and Theorem \ref{lm:conv.proj}(a), we have that $\{p^k\}$ is bounded) and let
$\{p^{k_j}\}$ be a subsequence of $\{p^k\}$ such that $p^{k_j}\rightharpoonup p^\infty$, i.e.,
\begin{align}
 \label{eq:903}
  z^{k_j}\rightharpoonup z^\infty\;\;\mbox{and}\;\; w_i^{k_j}\rightharpoonup w_i^{\infty},\quad i=1,\dots, n-1.
\end{align}
Using item (d), \eqref{eq:903} and the fact that $G_n=I$ (see Assumption (A2)), we obtain
\begin{align}
  \label{eq:16103}
  x_n^{k_j}\rightharpoonup z^\infty\;\;\mbox{and}\;\; y_i^{k_j}\rightharpoonup w_i^{\infty},\quad i=1,\dots, n-1.
\end{align}
Now define the maximal monotone operators $A:\HH_0\rightrightarrows \HH_0$,
$B:\HH_1\times\cdots\times\HH_{n-1}\rightrightarrows \HH_1\times\cdots\times\HH_{n-1}$
and the bounded linear operator $G:\HH_0\to \HH_1\times\cdots\times\HH_{n-1}$ by
\begin{align}
\label{eq:161011}
 A:= T_n,\quad B:=T_1\times \cdots \times T_{n-1}\;\;\mbox{and}\;\; G:=(G_1,\dots, G_{n-1}).
\end{align}
Using the above definitions of $A$ and $B$ and the inclusions in \eqref{eq:proxT_i}, we have
\begin{align}
 \label{eq:1610}
 a^{k_j}\in A(r^{k_j})\;\;\mbox{and}\;\; b^{k_j}\in B(s^{k_j}),
\end{align}
where
\begin{align}
 \label{eq:16102}
 a^{k_j}:=y_n^{k_j},\quad r^{k_j}:=x_n^{k_j},
 \quad b^{k_j}:=(y_1^{k_j},\dots, y_{n-1}^{k_j})\;\;\mbox{and}\;\; s^{k_j}:=(x_1^{k_j},\dots, x_{n-1}^{k_j}).
\end{align}
Moreover, \eqref{eq:16102} and \eqref{eq:16103} yield 
\begin{align}
 \label{eq:16104}
  r^{k_j}\rightharpoonup r^\infty\;\;\mbox{and}\;\; b^{k_j}\rightharpoonup b^\infty,
\end{align}
where
\begin{align}
 \label{eq:16105}
  r^\infty:= z^\infty\;\;\mbox{and}\;\; b^\infty:=(w_1^\infty,\cdots, w_{n-1}^\infty).
\end{align}
Note now that using \eqref{eq:16102}, the fact that $G^*(w_1,\dots, w_{n-1})=\sum_{i=1}^{n-1}G_i^*w_i$,
for all $(w_1,\cdots, w_{n-1})\in \HH_1\times\cdots\times\HH_{n-1}$, the fact that $G_n=I$ and the first limit in item (b), we find
\begin{align}
\label{eq:16106}
 a^{k_j} + G^* b^{k_j}  = \sum_{i=1}^{n}G_i^*y_i^{k_j}\to 0.
%
\end{align}
Using now the second limit in item (b) combined with \eqref{eq:16102} and the definition of $G$ in \eqref{eq:161011}, we obtain
\begin{align}
\label{eq:16109}
 G r^{k_j} - s^{k_j} \to 0.
\end{align}
Using Lemma \ref{lm:comb} below combined with \eqref{eq:1610}, \eqref{eq:16104}, \eqref{eq:16106} and \eqref{eq:16109} we 
conclude that
\begin{align*}
  b^\infty\in B(Gr^\infty)\;\;\mbox{and}\;\; -G^*b^\infty \in A(r^\infty),
\end{align*}
which, in turn, combined with \eqref{eq:161011} and \eqref{eq:16105} implies that
\begin{align*}
 w_i^\infty \in T_i(G_iz^\infty),\;\; i=1,\cdots, n-1,\;\; -\sum_{i=1}^{n-1}G_i^* w_i^\infty \in T_n(z^\infty),
\end{align*}
which means exactly (see \eqref{eq:KTS}) that $p^\infty = (z^\infty, w_1^\infty, \cdots, w_{n-1}^\infty)\in \mathcal{S}$. Hence, 
we conclude that every weak cluster point of $\{p^k\}$ belongs to $\mathcal{S}$.
\end{proof}

\mgap

Next is the first result on the asymptotic convergence of Algorithm \ref{alg:iPSM}.

\begin{theorem}[First result on the convergence of Algorithm \ref{alg:iPSM}]
 \label{th:conv01}
 Consider the sequences evolved by \emph{Algorithm \ref{alg:iPSM}} and let $\{p^k\}$ be as in \eqref{eq:def.pk2}.
 Assume that conditions \eqref{eq:901} and \eqref{eq:assum.rho2} of \emph{Proposition \ref{pr:varios}} hold, i.e., assume that
 \begin{align}
   \label{eq:1423}
  \sum_{k=0}^\infty\,\alpha_k\norm{p^k-p^{k-1}}^2_\gamma<\infty
 \end{align}
and, for $i=1,\dots, n$,
\begin{align}
 \label{eq:161012}
 0<\underline{\rho}\leq \rho_i^k\leq \overline{\rho}<\infty\qquad \forall k\geq 0.
\end{align}
Then, there exists $(z^\infty, w_1^\infty, \cdots, w_{n-1}^\infty)\in \mathcal{S}$ such that
$z^k \rightharpoonup z^\infty$ and
$w_i^k\rightharpoonup w_i^\infty$, for $i=1,\dots, n-1$.
Furthermore, 
$x_i^k\rightharpoonup G_i z^\infty$ and
$y_i^k\rightharpoonup w_i^\infty$, for $i=1,\dots, n$, where $w_n^k$ is as in \eqref{eq:w_nk}.
\end{theorem}
\begin{proof}
In view of Propositions \ref{prop:Q1} and \ref{pr:varios}(e) and Theorem \ref{lm:conv.proj}(b) one concludes that
that $\{p^k\}$ converges weakly to some $p^\infty:=(z^\infty, w_1^\infty, \cdots, w_{n-1}^\infty)$ in 
$\mathcal{S}$ as in \eqref{eq:KTS}. Using the definition of $p^k$ in \eqref{eq:def.pk2} one easily concludes
that $z^k \rightharpoonup z^\infty$ and $w_i^k\rightharpoonup w_i^\infty$, for $i=1,\dots, n-1$, which 
in turn combined with Proposition \ref{pr:varios}(d) implies that $x_i^k\rightharpoonup G_i z^\infty$ and
$y_i^k\rightharpoonup w_i^\infty$, for $i=1,\dots, n$.
\end{proof}

\mgap

Next theorem shows the convergence of Algorithm \ref{alg:iPSM} under certain assumptions on $\alpha$, $\overline{\beta}$ and the
sequence $\{\alpha_k\}$ (see the remarks below).

\begin{theorem}[Second result on the convergence of Algorithm \ref{alg:iPSM}]
 \label{th:conv02}
Consider the sequences evolved by \emph{Algorithm \ref{alg:iPSM}} and assume that
$\alpha\in [0,1)$, $\overline{\beta}\in (0,2)$  and $\{\alpha_k\}$ satisfy \emph{(}for some $\overline{\alpha}>0$\emph{)} the conditions
\eqref{eq:alpha_k} and \eqref{eq:beta(alpha)} of \emph{Theorem \ref{thm:second.conv}}, i.e.,
\begin{align}
      0\leq \alpha_k\leq \alpha_{k+1}\leq \alpha<\overline{\alpha}<1\qquad \forall k\geq 0
  \end{align}
 and
  \begin{align}
  \label{eq:beta_alpha_4}
    \overline{\beta}=\overline{\beta}(\overline{\alpha}):=\dfrac{2(\overline{\alpha}-1)^2}
    {2(\overline{\alpha}-1)^2+3\overline{\alpha}-1}.
  \end{align}
 Assume also that condition \eqref{eq:161012} holds, i.e., assume that, for $i=1,\dots, n$,
 \begin{align}
 0<\underline{\rho}\leq \rho_i^k\leq \overline{\rho}<\infty\qquad \forall k\geq 0.
\end{align}
 Then, the same conclusions of \emph{Theorem \ref{th:conv01}} hold, i.e.,  
 there exists $(z^\infty, w_1^\infty, \cdots, w_{n-1}^\infty)\in \mathcal{S}$ such that
$z^k \rightharpoonup z^\infty$ and
$w_i^k\rightharpoonup w_i^\infty$, for $i=1,\dots, n-1$.
Furthermore, 
$x_i^k\rightharpoonup G_i z^\infty$ and
$y_i^k\rightharpoonup w_i^\infty$, for $i=1,\dots, n$, where $w_n^k$ is as in \eqref{eq:w_nk}.
 \end{theorem}
\begin{proof}
In view of Propositions \ref{prop:Q1} and \ref{pr:varios}(e) and Theorem \ref{thm:second.conv}(b) one concludes that
that $\{p^k\}$ converges weakly to some $p^\infty:=(z^\infty, w_1^\infty, \cdots, w_{n-1}^\infty)$ in 
$\mathcal{S}$ as in \eqref{eq:KTS}. The rest of the proof follows the same argument used in Theorem \ref{th:conv01}'s proof.
\end{proof}

\mgap

\noindent
{\bf Remarks.}

\begin{itemize}
 \item[\mbox{(i)}] We emphasize that the conditions on $\alpha$, $\overline{\beta}$ and on the sequence $\{\alpha_k\}$ above are exactly
 the same of Theorem \ref{thm:second.conv}, namely \eqref{eq:alpha_k} and \eqref{eq:beta(alpha)}. See also the second remark following
 Theorem \ref{thm:second.conv} and Figure \ref{fig02} for a discussion of the interplay between inertial and overrelaxation 
 parameters.
 \item[\mbox{(ii)}] Note that, since $(z^\infty, w_1^\infty, \cdots, w_{n-1}^\infty)\in \mathcal{S}$ in Theorem \ref{th:conv02}, it follows that the weak limit $z^\infty \in \HH_0$ is a solution of the monotone inclusion problem \eqref{eq:Pmip1}.
\end{itemize}

\section{Numerical experiments}
 \label{sec:ne}


In this section we present simple numerical experiments on $\ell_1$--regularized least square problems
\begin{align}
\label{eq:lasso}
\min_{x\in \R^d}\,\left\{\dfrac{1}{2}\norm{Qx-b}_2^2 + \lambda \norm{x}_1\right\},
\end{align}
where $Q\in \R^{m\times d}$, $b\in \R^m$ and $\lambda\geq 0$.  
Let $\mathcal{R}=\{R_1,\dots, R_r\}$ be an arbitrary partition
\footnote{$R_i\neq \emptyset\; (i=1,\dots, r)$, $R_i\cap R_j=\emptyset$ for $i\neq j$ and $\cup_{i=1}^r\,R_i = \{1,\dots, m\}$.} 
of $\{1,\dots, m\}$
and, for $i=1,\dots, r$, let $Q_i\in \R^{|R_i|\times d}$ be the submatrix of $Q$ with rows corresponding to indices
in $R_i$ and similarly let $b_i\in \R^{|R_i|}$ be the corresponding subvector of $b$.
Then, problem \eqref{eq:lasso} is equivalent to the minimization problem
\begin{align*}
 %
 \min_{x\in \R^d}\,\left\{\sum_{i=1}^r\,\dfrac{1}{2}\norm{Q_i x-b_i}_{2}^2 + \lambda \norm{x}_1\right\},
\end{align*}
which, in turn, is clearly equivalent to the monotone inclusion problem
\begin{align}
 \label{eq:lasso03}
   0\in \sum_{i=1}^r\,Q_i^T(Q_i x - b_i) + \partial (\lambda \norm{x}_1).
\end{align} 
On the other hand, \eqref{eq:lasso03} is a special instance of the monotone inclusion problem \eqref{eq:Pmip1} with $z\leftarrow x$, $n=r+1$, $G_i=I$ ($i=1,\dots, n$), 
\begin{align*}
 T_i(x)=Q_i^T(Q_i x -b_i)\;( i=1,\dots, n-1)\;\;\mbox{and}\;\;  T_n(x)=\partial (\lambda \norm{x}_1).
\end{align*}

In this section, we shall apply Algorithm \ref{alg:iPSM} for solving \eqref{eq:lasso03} (and, in particular, \eqref{eq:lasso}) with the following choice of
parameters (see Steps 0, 1, 2 and 4 of Algorithm \ref{alg:iPSM}):
\begin{center}
$\alpha_k\equiv \alpha=0.1$,\; $\sigma=0.99$,\; $\gamma=1$,\; $\rho_i^k\equiv 1$\;\;\mbox{and}\;\;
$\beta_k\equiv \underline{\beta}=\overline{\beta} = 1.5519$.
\end{center}
The value $1.5519$ is computed from \eqref{eq:beta_alpha_4} with $\overline{\alpha}=0.17>\alpha$.  
Following \cite{Johnst.Eckst.2018},  we stop
the algorithm using the stopping criterion
\begin{align}
\label{eq:stop_crit}
%
%
 \dfrac{|F(z^k) - F^*|}{F^*}\leq 10^{-4},
\end{align}
where $F(\cdot)$ denotes the objective function in \eqref{eq:lasso} and $F^*$ is the optimal value of the 
problem  estimated by running Algorithm \ref{alg:iPSM} at least $10^4$ iterations and taking the minimum objective value.  

At each iteration $k$ of Algorithm \ref{alg:iPSM}, we used two different strategies for computing $(x_i^k,y_i^k)$ ($i=1,\dots, n$) satisfying \eqref{eq:proxT_i}: for $i=1,\dots, n-1$, in which case $ T_i(x)=Q_i^T(Q_i x -b_i)$, we implemented the
standard conjugate gradient (CG) algorithm for computing $x=x_i^k$ as an approximate solution of the linear system 
\begin{align*}
 (Q_i^T Q_i  + I) x =\widehat z^k + \widehat w_i^k + Q_i^T b_i
\end{align*}
until the satisfaction of the relative-error condition in \eqref{eq:proxT_i} with $y_i^k:=T_i(x_i^k)$ by the residual $e_i^k:=T_i(x_i^k)+x_i^k - (\widehat z^k + \widehat w_i^k)$. On the other hand, for $i=n$, in which case $T_n(x)=\partial (\lambda \norm{x}_1)$, we
set $x_i^k = \mbox{prox}_{\lambda \|\cdot\|_1}(\widehat z^k + \widehat w_i^k)$ and
$y_i^k = (\widehat z^k + \widehat w_i^k) - x_i^k$ (in this case, $e_i^k=0$).

\mgap
\mgap

\noindent
{\bf Data sets}.
We implemented Algorithm \ref{alg:iPSM} using the following data sets:
\begin{itemize}
\item Four randomly generated instances of \eqref{eq:lasso}: RandomA, RandomB, RandomC and RandomD. We used the Matlab command ``randn'' to generate $Q\in \R^{m\times d}$, and $b\in \R^m$ with 
$b_j \in \{0,1\}\; (j=1,\dots, m)$ where $b=(b_1,\dots, b_j, \dots, b_m)$ (see Table \ref{tab:random}).
\begin{table}
\caption{Dimensions of $Q\in \R^{m\times d}$ and $b\in \R^m$, size $r$ of the partition $\mathcal{R}$ of $\{1,\dots, m\}$ and number of rows of each submatrix of $Q$ on 
four randomly generated instances of \eqref{eq:lasso}}
\begin{center}
\begin{tabular}{cc|c|c|c|l}
\cline{2-5}
&  \multicolumn{1}{ |c| }{\multirow{2}{*}{$m$}} & \multicolumn{1}{ |c| }{\multirow{2}{*}{$d$}} & \multicolumn{1}{ |c| }{\multirow{2}{*}{$r$}} & \multicolumn{1}{ |c| }{\multirow{2}{*}{$ \mid R_i \mid $}} \\ 
 & \multicolumn{1}{ |c|  }{} & \multicolumn{1}{ |c|  }{} & \multicolumn{1}{ |c | }{} &    \multicolumn{1}{ |c | }{}     \\\cline{1-5}
\multicolumn{1}{ |c | }{\multirow{2}{*}{RandomA} }  & \multicolumn{1}{ |c| }{\multirow{2}{*}{1000}} & \multicolumn{1}{ |c| }{\multirow{2}{*}{1000}} & \multicolumn{1}{ |c| }{\multirow{2}{*}{10}} & \multicolumn{1}{ |c| }
{\multirow{2}{*}{\hspace{-0.2cm}$100$ $\; (i=1, \dots , 10)$}} &     \\
\multicolumn{1}{ |c|  }{}                        &
\multicolumn{1}{ |c| }{}  & \multicolumn{1}{ |c| }{} &  \multicolumn{1}{ |c| }{} & \multicolumn{1}{ |c| }{} &     \\ \cline{1-5}
\multicolumn{1}{ |c|  }{\multirow{2}{*}{RandomB} }  & \multicolumn{1}{ |c| }{\multirow{2}{*}{5000}} & \multicolumn{1}{ |c| }{\multirow{2}{*}{100}} & \multicolumn{1}{ |c| }{\multirow{2}{*}{20}} & \multicolumn{1}{ |c| }
{\multirow{2}{*}{\hspace{-0.2cm}$250$ $\; (i=1, \dots , 20)$}} &     \\ 
\multicolumn{1}{ |c | }{}                        &
\multicolumn{1}{ |c| }{} & \multicolumn{1}{ |c| }{} & \multicolumn{1}{ |c| }{} & \multicolumn{1}{ |c| }{} &     \\ \cline{1-5}
\multicolumn{1}{ |c | }{\multirow{2}{*}{RandomC} }  & \multicolumn{1}{ |c| }{\multirow{2}{*}{50000}} & \multicolumn{1}{ |c| }{\multirow{2}{*}{100}} & \multicolumn{1}{ |c| }{\multirow{2}{*}{250}} & \multicolumn{1}{ |c| }{\multirow{2}{*}{$200$ $ \; (i=1, \dots , 250)$}} &     \\ 
\multicolumn{1}{ |c | }{}                        &
\multicolumn{1}{ |c| }{} & \multicolumn{1}{ |c| }{} & \multicolumn{1}{ |c| }{} & \multicolumn{1}{ |c| }{} &     \\ \cline{1-5}
\multicolumn{1}{ |c | }{\multirow{2}{*}{RandomD} }  & \multicolumn{1}{ |c| }{\multirow{2}{*}{100000}} & \multicolumn{1}{ |c| }{\multirow{2}{*}{100}} & \multicolumn{1}{ |c| }{\multirow{2}{*}{325}} & \multicolumn{1}{ |c| }{$307$ $ \; (i=1, \dots , 324)$} &     \\ \cline{5-5}
\multicolumn{1}{ |c|  }{}                        &
\multicolumn{1}{ |c| }{} & \multicolumn{1}{ |c| }{} & \multicolumn{1}{ |c| }{} & \multicolumn{1}{ |c| }
{\hspace{-1cm}$532$ $ \; (i=325)$} &     \\ \cline{1-5}
\end{tabular}
\end{center}
\label{tab:random}
\end{table}
%
\item Five data sets (real examples) from the UCI Machine Learning Repository \cite{Dua:2019}: 
 the blog feedback dataset (BlogFeedback)
\footnote{https://archive.ics.uci.edu/ml/datasets/BlogFeedback.}
,  
communities and crime dataset (Crime)
\footnote{http://archive.ics.uci.edu/ml/datasets/communities+and+crime.}
, 
DrivFace dataset (DrivFace)
\footnote{https://archive.ics.uci.edu/ml/datasets/DrivFace.}
,
Single-Pixed Camera (Mug32)
\footnote{see \cite{alv.eck.ger.mel-preprint19}.}
and Breast Cancer Wisconsin (Diagnostic)
dataset (Wisconsin) \footnote{https://archive.ics.uci.edu/ml/datasets/Breast+Cancer+Wisconsin+(Diagnostic).}
(see Table \ref{tab:exampleUCI}).
\begin{table}
\caption{Dimensions of $Q\in \R^{m\times d}$ and $b\in \R^m$, size $r$ of the partition $\mathcal{R}$ of $\{1,\dots, m\}$ and number of rows of each submatrix of $Q$ on five real examples (from the UCI Machine Learning Repository \cite{Dua:2019}) of \eqref{eq:lasso}}
\begin{center}
\begin{tabular}{cc|c|c|c|l}
\cline{2-5}
&  \multicolumn{1}{ |c| }{\multirow{2}{*}{$m$}} & \multicolumn{1}{ |c| }{\multirow{2}{*}{$d$}} & \multicolumn{1}{ |c| }{\multirow{2}{*}{$r$}} & \multicolumn{1}{ |c| }{\multirow{2}{*}{$\mid R_i\mid$}} \\ 
 & \multicolumn{1}{ |c | }{} & \multicolumn{1}{ |c|  }{} & \multicolumn{1}{ |c|  }{} &    \multicolumn{1}{ |c | }{}     \\\cline{1-5}
\multicolumn{1}{ |c | }{\multirow{2}{*}{BlogFeedback} }  & \multicolumn{1}{ |c| }{\multirow{2}{*}{60021}} & \multicolumn{1}{ |c| }{\multirow{2}{*}{280}} & \multicolumn{1}{ |c| }{\multirow{2}{*}{175}} & \multicolumn{1}{ |c| }{$343$ $ \; (i=1, \dots,174)$} &     \\ \cline{5-5}
\multicolumn{1}{ |c|  }{}                        &
\multicolumn{1}{ |c| }{}  & \multicolumn{1}{ |c| }{} &  \multicolumn{1}{ |c| }{} & \multicolumn{1}{ |c| }
{\hspace{-1.08cm}$339$ $ \; (i=175)$} &     \\ \cline{1-5}
\multicolumn{1}{ |c|  }{\multirow{2}{*}{Crime} }  & \multicolumn{1}{ |c| }{\multirow{2}{*}{1994}} & \multicolumn{1}{ |c| }{\multirow{2}{*}{121}} & \multicolumn{1}{ |c| }{\multirow{2}{*}{10}} & \multicolumn{1}{ |c| }
{\hspace{-0.4cm}$200$ $ \; (i=1, \dots, 9)$} &     \\ \cline{5-5}
\multicolumn{1}{ |c | }{}                        &
\multicolumn{1}{ |c| }{} & \multicolumn{1}{ |c| }{} & \multicolumn{1}{ |c| }{} & \multicolumn{1}{ |c| }
{\hspace{-1.3cm}$194$ $ \; (i=10)$} &     \\ \cline{1-5}
\multicolumn{1}{ |c |}{\multirow{2}{*}{DrivFace} }  & \multicolumn{1}{ |c| }{\multirow{2}{*}{606}} & \multicolumn{1}{ |c| }{\multirow{2}{*}{6400}} & \multicolumn{1}{ |c| }{\multirow{2}{*}{6}} & \multicolumn{1}{ |c|}{\multirow{2}{*}
{\hspace{-0.406cm}$101$ $ \; (i=1,\dots,6)$}} &     \\ 
\multicolumn{1}{ |c|  }{}                        &
\multicolumn{1}{ |c| }{} & \multicolumn{1}{ |c| }{} & \multicolumn{1}{ |c| }{} & \multicolumn{1}{ |c| }{} &     \\ \cline{1-5}
\multicolumn{1}{ |c|  }{\multirow{2}{*}{Mug32} }  & \multicolumn{1}{ |c| }{\multirow{2}{*}{410}} & \multicolumn{1}{ |c| }{\multirow{2}{*}{1024}} & \multicolumn{1}{ |c| }{\multirow{2}{*}{4}} & \multicolumn{1}{ |c| }
{\hspace{-0.7cm}$100$ $ \; (i=1,2,3)$} &     \\ \cline{5-5}
\multicolumn{1}{ |c | }{}                        &
\multicolumn{1}{ |c| }{} & \multicolumn{1}{ |c| }{} & \multicolumn{1}{ |c| }{} & \multicolumn{1}{ |c| }
{\hspace{-1.408cm}$110$ $ \; (i=4)$} &     \\ \cline{1-5}
\multicolumn{1}{ |c|  }{\multirow{2}{*}{Wisconsin} }  & \multicolumn{1}{ |c| }{\multirow{2}{*}{198}} & \multicolumn{1}{ |c| }{\multirow{2}{*}{30}} & \multicolumn{1}{ |c| }{\multirow{2}{*}{3}} & \multicolumn{1}{ |c| }{\multirow{2}{*}
{\hspace{-0.8cm}$66$ $\;(i=1,2,3)$}} &     \\ 
\multicolumn{1}{ |c | }{}                        &
\multicolumn{1}{ |c| }{} & \multicolumn{1}{ |c| }{} & \multicolumn{1}{ |c| }{} & \multicolumn{1}{ |c| }{} &     \\ \cline{1-5}
\end{tabular}
\end{center}
\label{tab:exampleUCI}
\end{table} 
\end{itemize}
We also used $\lambda = 0.1\norm{Q^T b}_\infty$ (see \cite{boy.par.chu-dis.ftml11}) in
\eqref{eq:lasso}. Table \ref{lasso_outer} shows the number of outer iterations, and Table \ref{lasso_runtime}
shows runtimes in seconds.
Figures \ref{figura:graf} and \ref{figura:graf02} show the same results graphically (see the stopping criterion \eqref{eq:stop_crit}).

\begin{table}
\caption{Outer iterations for LASSO problems} 
\vspace{0.3cm}
\centering
 \begin{tabular}{llrrcrrcrrcc}\hline \hline \\
   Problem & & & & PS & & & PS\_in\_rel & & & $ \dfrac{iteration 2}{ iteration 1} $ &\\ 
     & & & & \tiny $iteration 1$ & & & \tiny $iteration 2$ & & & & \\ \hline \\
     BlogFeedback & & & & 2968 & & & \textbf{2342}  & & & 0.7891  & \\
     Crime        & & & & \textbf{211} & & & 216 & & & 1.0237  & \\
     DrivFace     & & & & 2008 & & & \textbf{585} & & & 0.2913  & \\
     Mug32        & & & & 203 & & & \textbf{192} & & &  0.9458 & \\
     Wisconsin    & & & & 211 & & & \textbf{210} & & & 0.9952  & \\
     RandomA     & & & & 219 & & & \textbf{185} & & & 0.8447  & \\
     RandomB     & & & & 23 & & & \textbf{21} & & &  0.9131 & \\
     RandomC     & & & & 408 & & & \textbf{151} & & & 0.3701  & \\
     RandomD   & & & & 507 & & & \textbf{278} & & &  0.5483 & \\ \hline \\
     Geometric mean & & & & 337.04 & & & \textbf{231.98} & & &  0.6883 & \\ \hline \hline \\
 \end{tabular}
\label{lasso_outer}
\end{table}
\begin{table}
\caption{LASSO runtimes in seconds} 
\vspace{0.3cm}
\centering
 \begin{tabular}{llrrcrrcrrcc}\hline \hline \\
   Problem & & & & PS & & & PS\_in\_relerr & & & $ \dfrac{time 2}{ time 1} $ &\\ 
     & & & & \tiny $time 1$ & & & \tiny $time 2$ & & & & \\ \hline \\
     BlogFeedback & & & & 207.18 & & & \textbf{130.44} & & & 0.6296  & \\
     Crime        & & & & 0.85 & & & \textbf{0.78} & & & 0.9176  & \\
     DrivFace     & & & & 133.19 & & & \textbf{37.11} & & &  0.2786 & \\
     Mug32        & & & & 1.36 & & & \textbf{1.18} & & &  0.8676 & \\
     Wisconsin    & & & & 0.15 & & & \textbf{0.11} & & &  0.7333 & \\
     randomA     & & & & 2.45 & & & \textbf{1.69} & & &  0.6898 & \\
     randomB     & & & & 0.25 & & & \textbf{0.13} & & & 0.52  & \\
     randomC     & & & & 10.53 & & & \textbf{4.08} & & &  0.3875 & \\
     randomD  & & & & 20.09 & & & \textbf{11.31} & & &  0.5629 & \\ \hline \\
     Geometric mean & & & & 3.79 & & & \textbf{2.57} & & & 0.6793  & \\ \hline \hline \\
 \end{tabular}
\label{lasso_runtime}
\end{table}
\begin{figure}
    \centering
    \caption{Comparison of performance in LASSO problems}
    \label{figura:graf}
    \vspace{0.2cm}
    \begin{minipage}{\linewidth}
    \centering
    \subfloat{
    \includegraphics[scale=0.32]{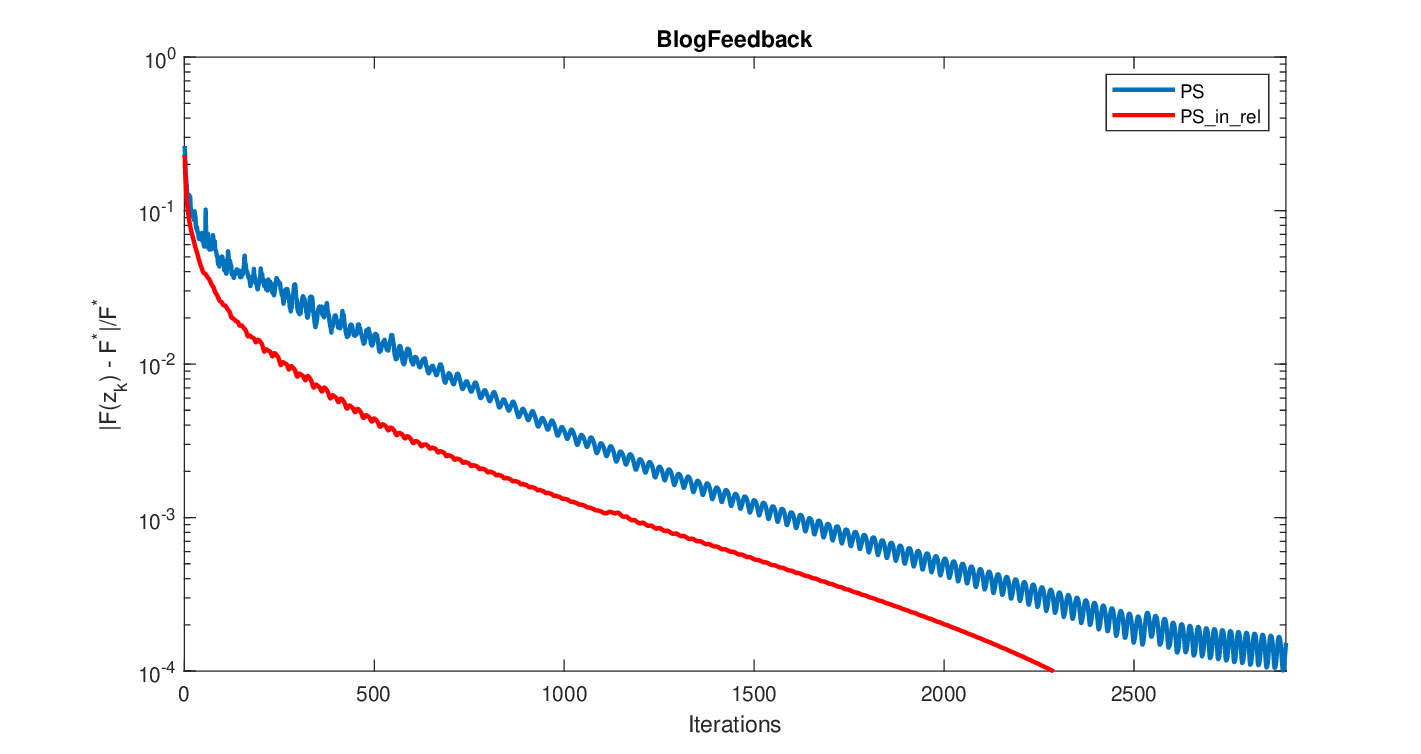} } 
    ~
    \subfloat{
    \includegraphics[scale=0.32]{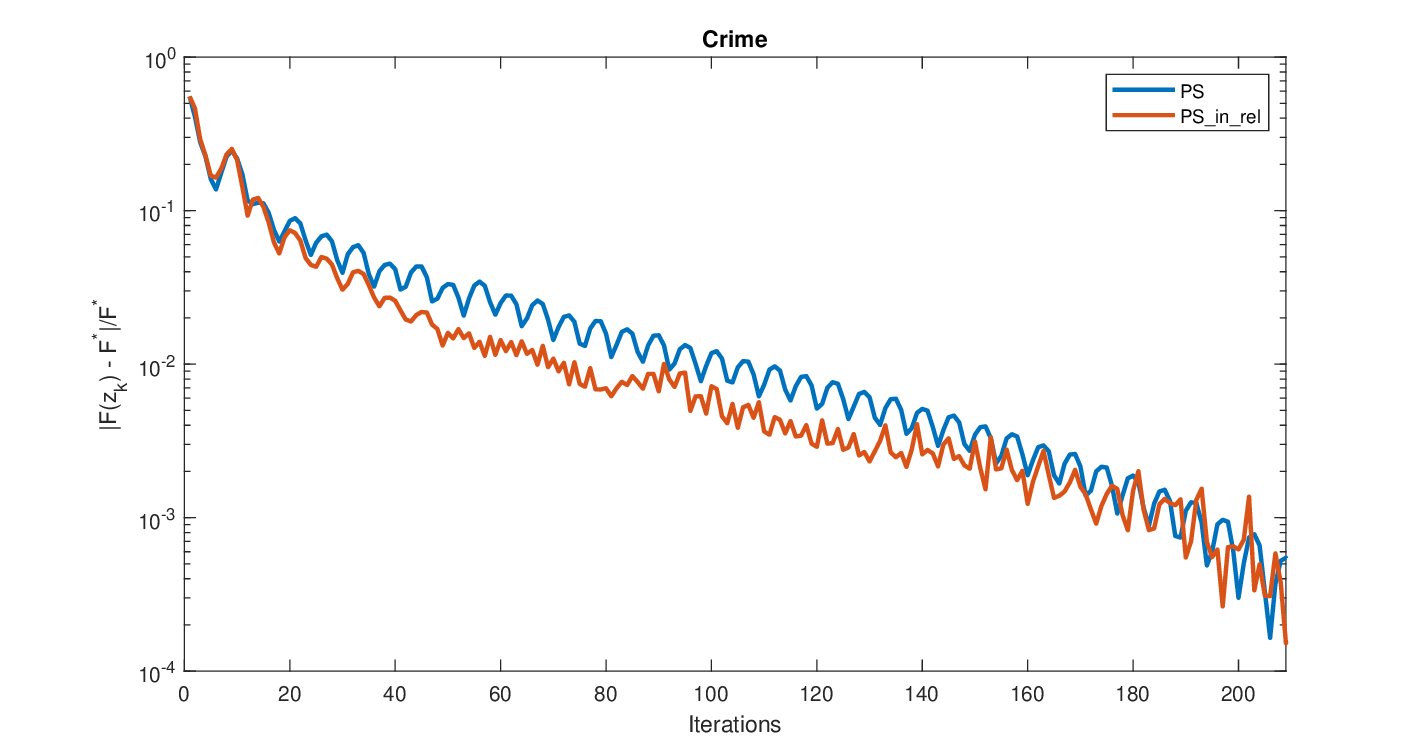}}
    \end{minipage}\par\medskip
    \begin{minipage}{\linewidth}
    \centering
    \subfloat{
    \includegraphics[scale=0.32]{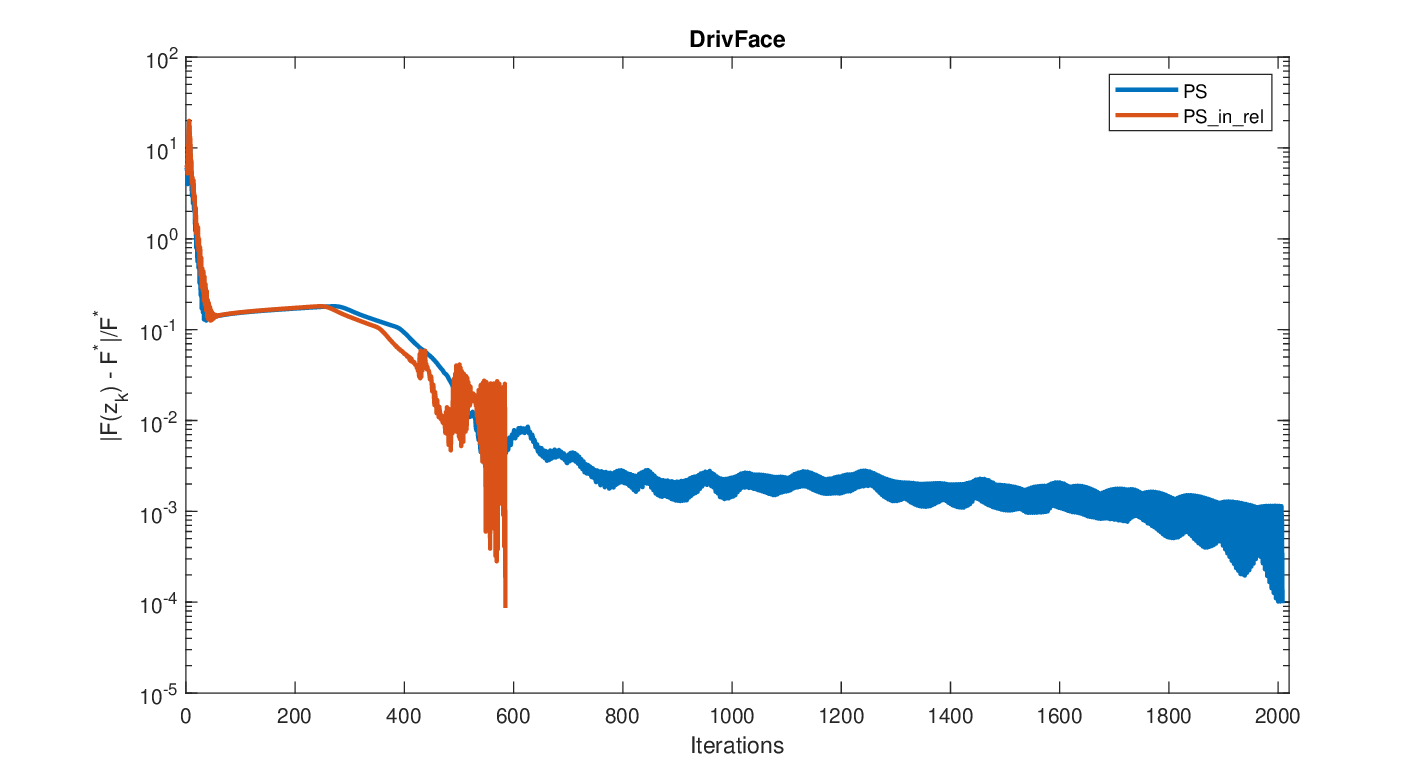}} 
    ~
    \subfloat{
    \includegraphics[scale=0.32]{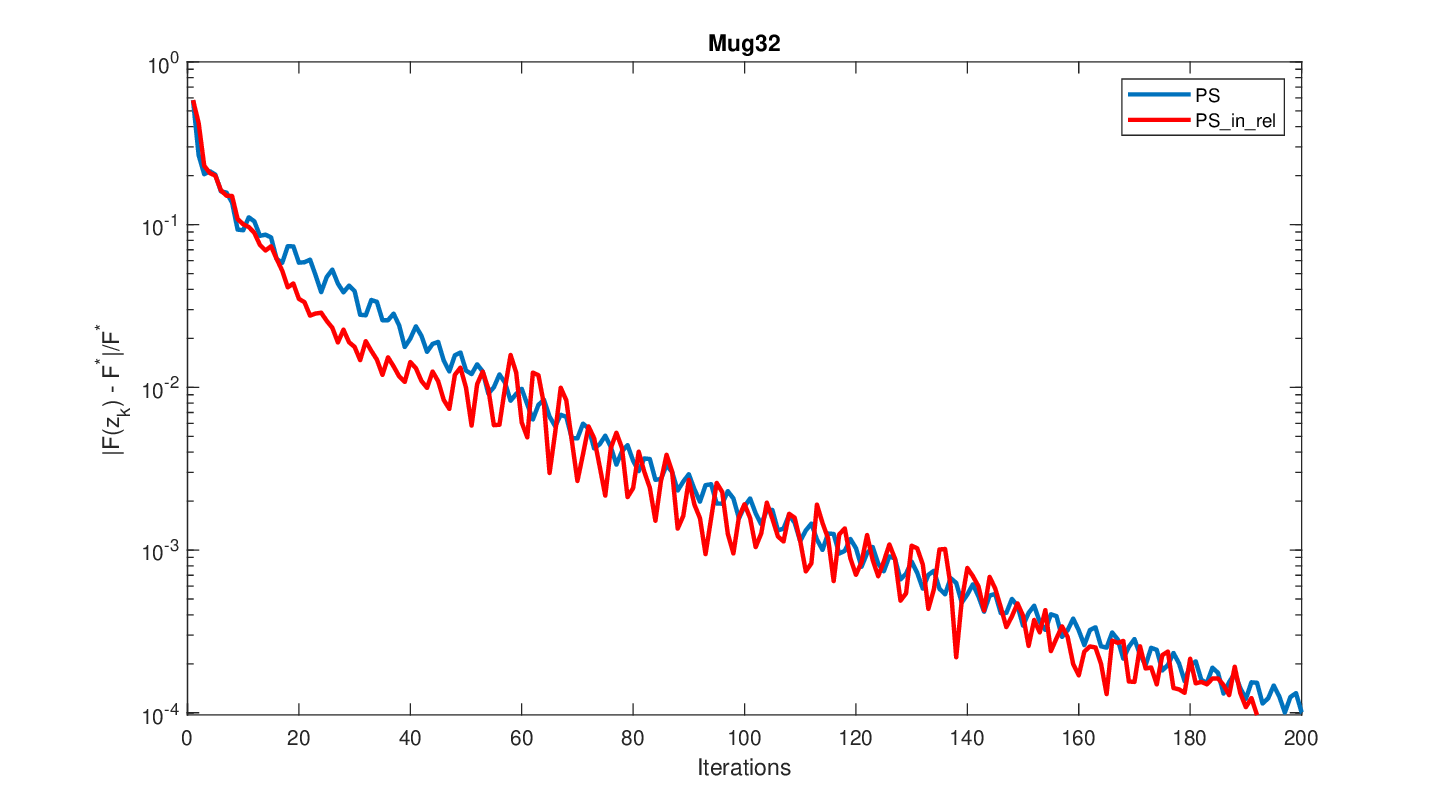}}
    \end{minipage}\par\medskip
    \begin{minipage}{\linewidth}
    \centering
    \subfloat{
    \includegraphics[scale=0.32]{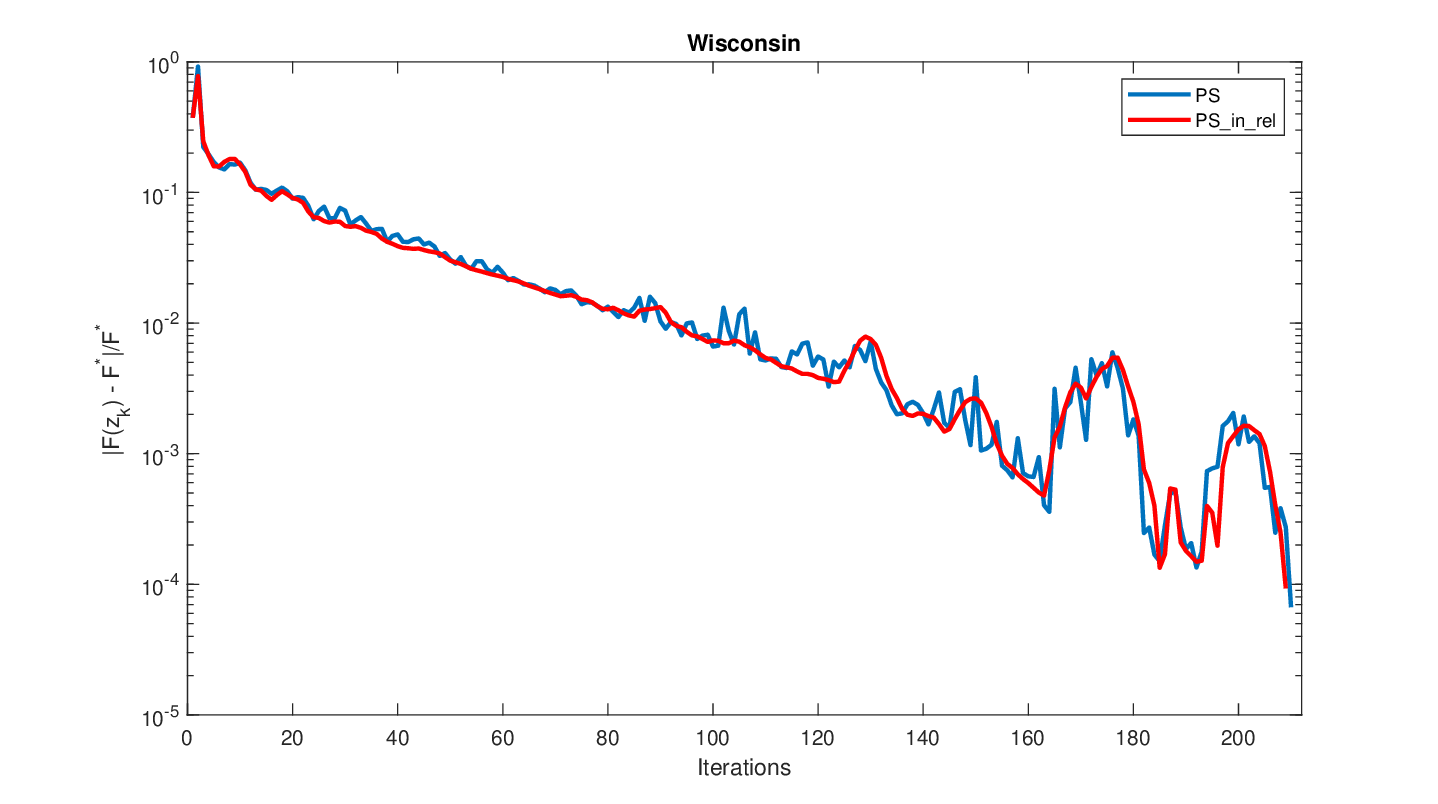}} 
    \end{minipage}
    %
    %
    \end{figure}

\begin{figure}
    \centering
    \caption{Comparison of performance in LASSO problems}
    \label{figura:graf02}
    \vspace{0.2cm}
\begin{minipage}{\linewidth}
    \centering
    \subfloat{
    \includegraphics[scale=0.32]{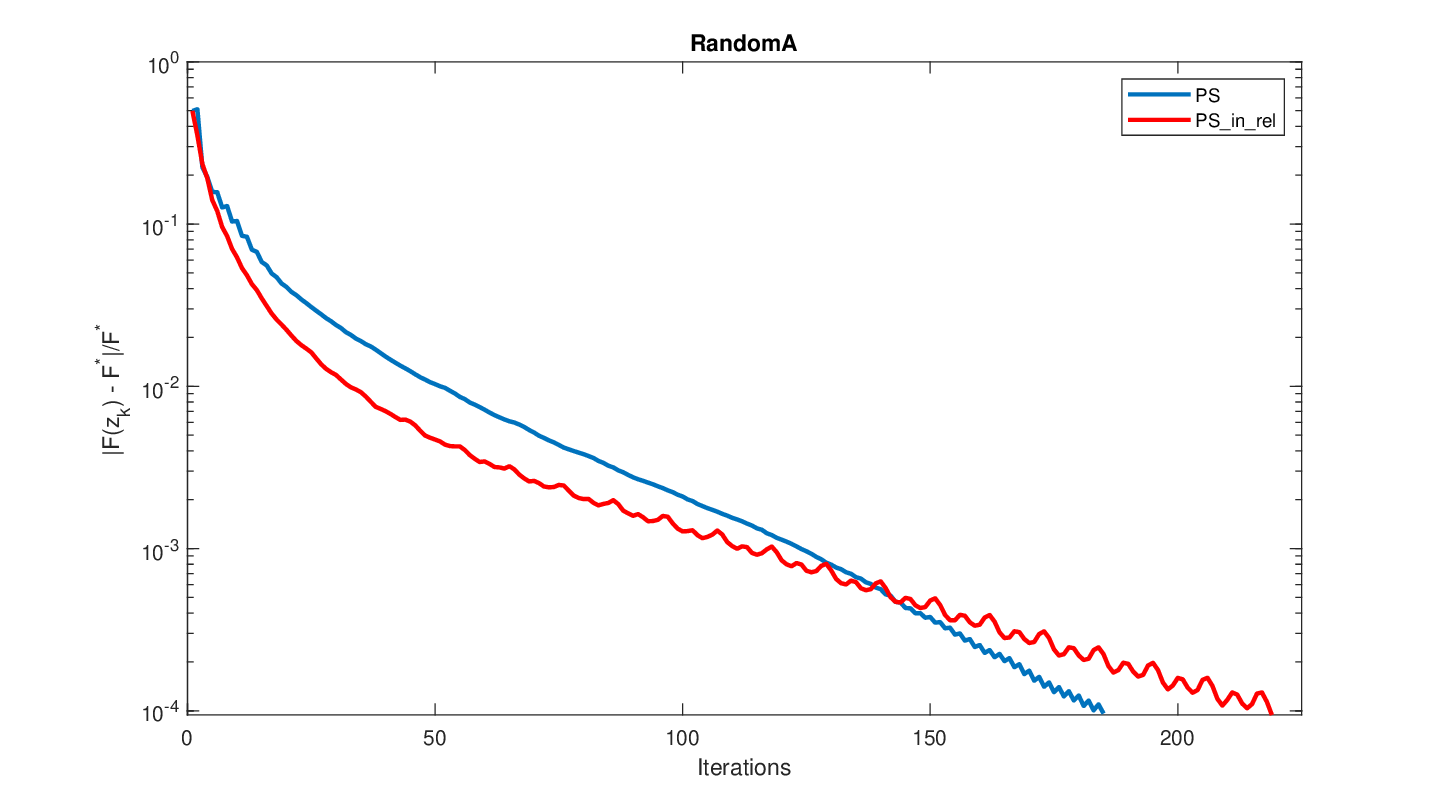}} 
    ~
    \subfloat{
    \includegraphics[scale=0.32]{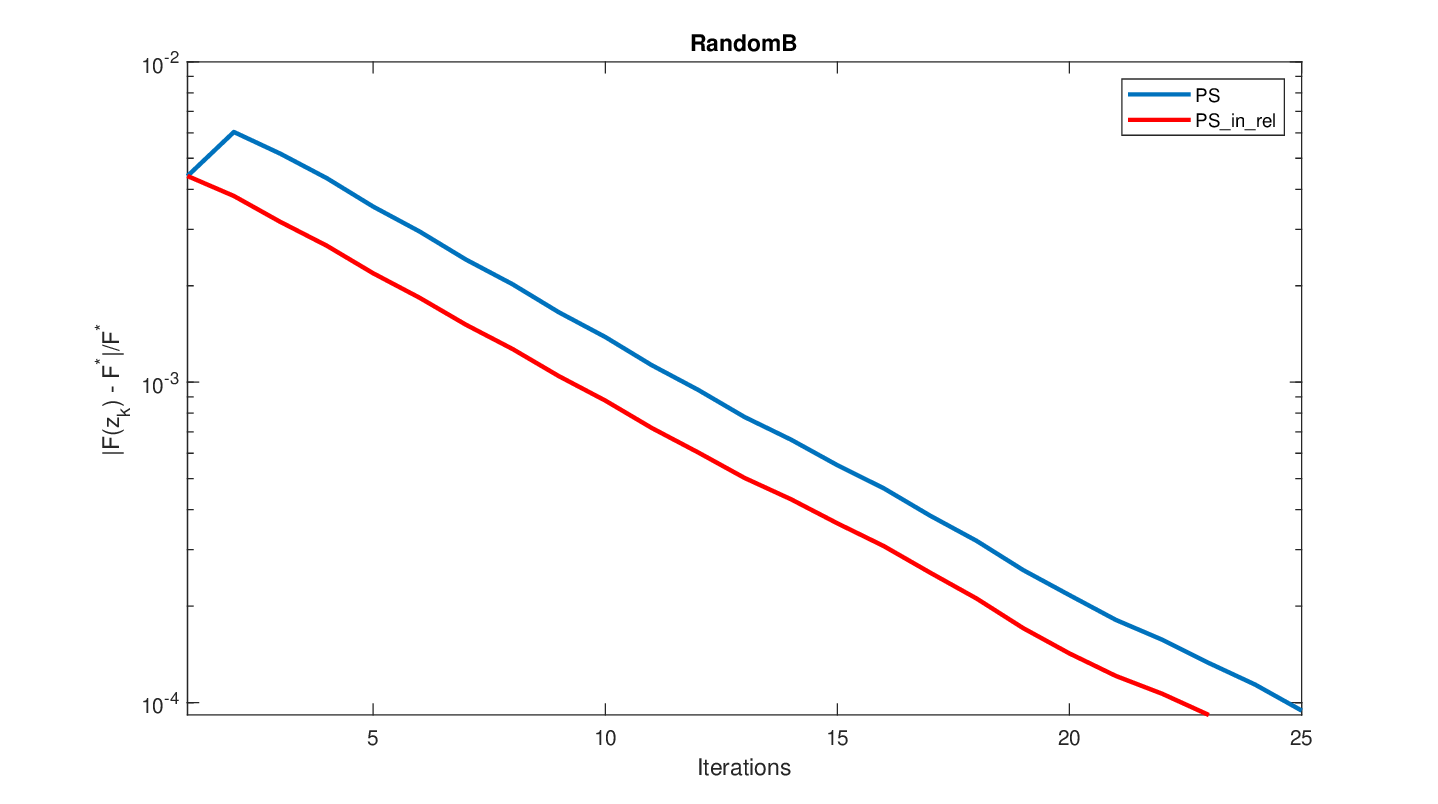}}
    \end{minipage}\par\medskip
    \begin{minipage}{ \linewidth}
    \centering
   \subfloat{
    \includegraphics[scale=0.32]{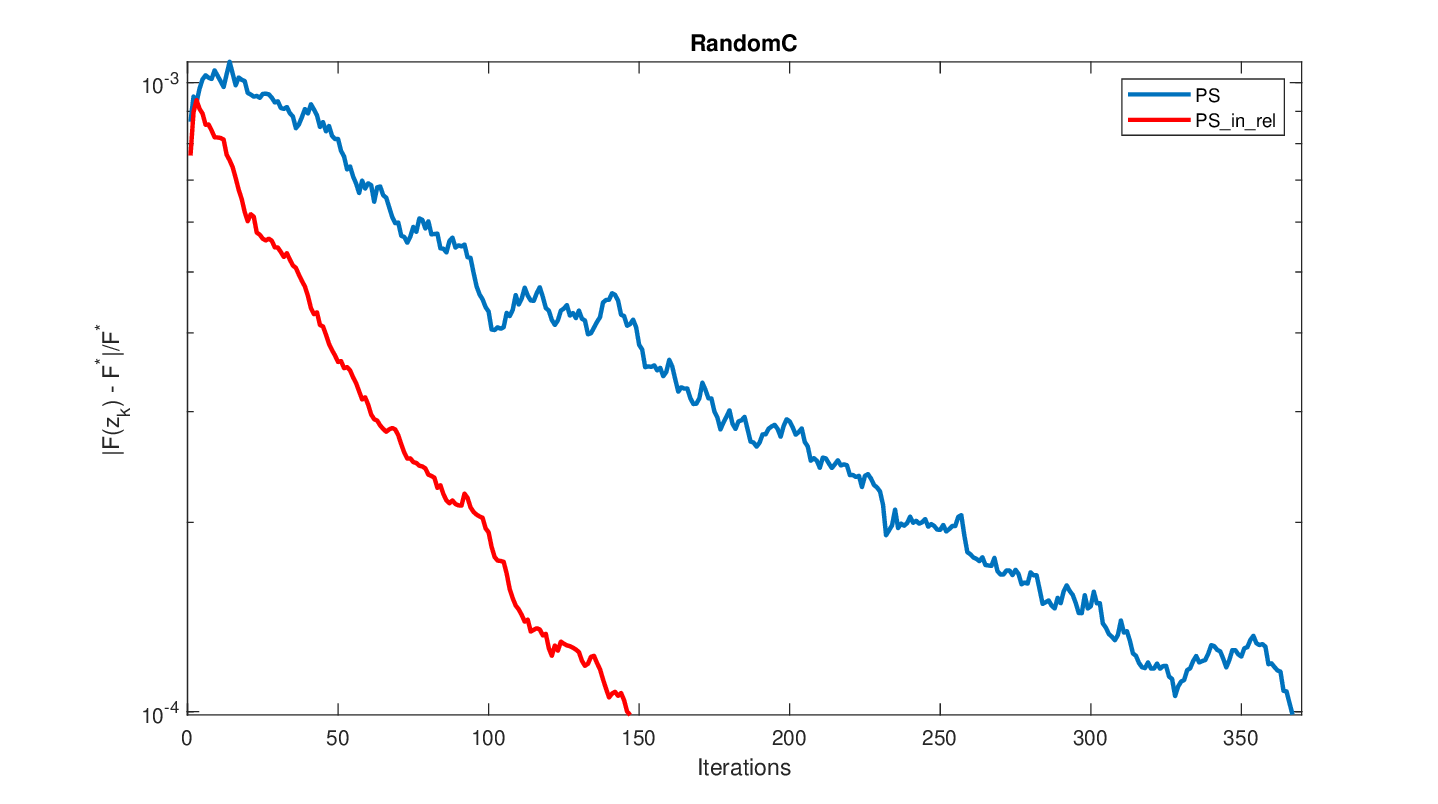}}
    ~
    \subfloat{
    \includegraphics[scale=0.32]{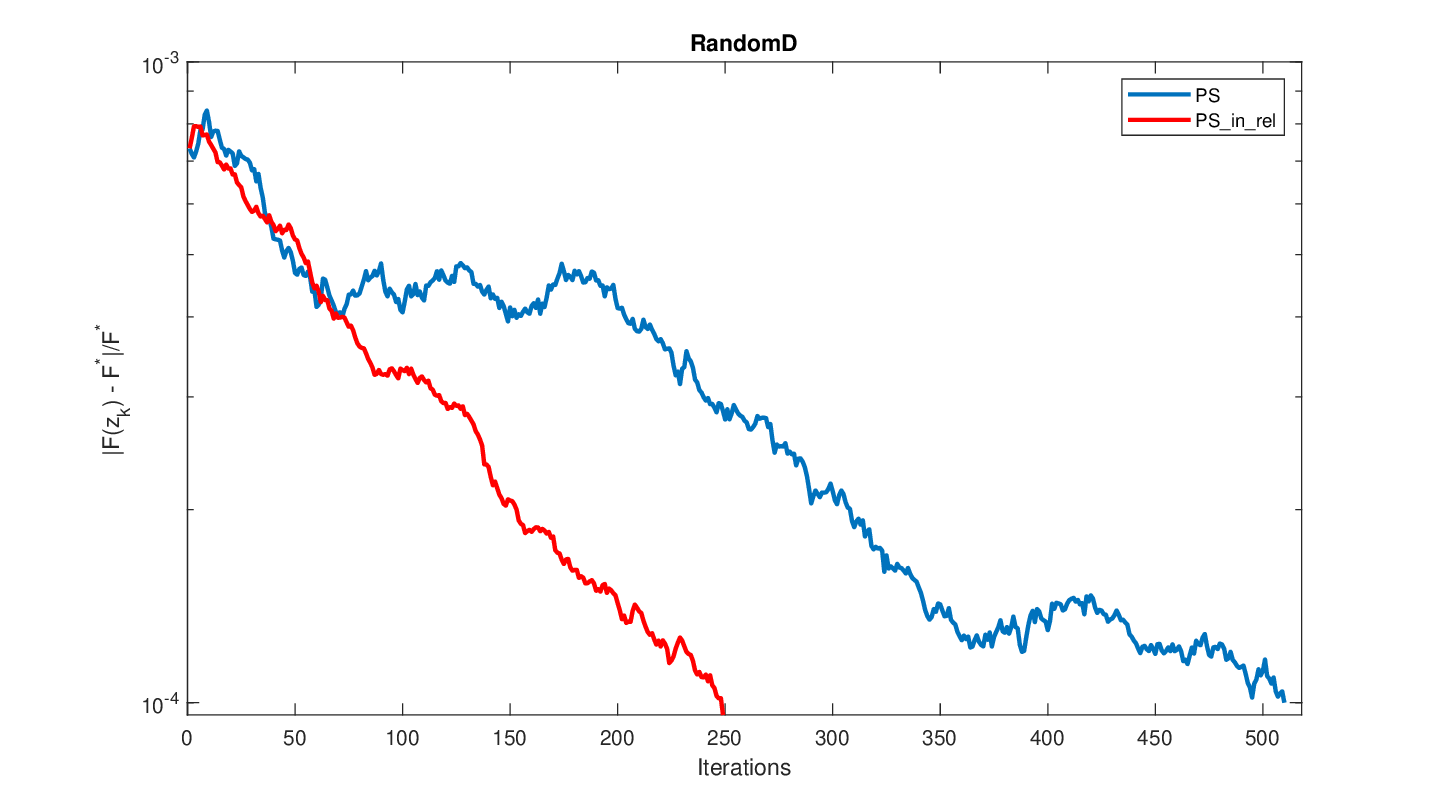}}
    \end{minipage}
    \end{figure}

\appendix
\section{Auxiliary results}

The following lemma was essentially proved by Alvarez and Attouch in  \cite[Theorem 2.1]{alv.att-iner.svva01} (see also
\cite[Lemma A.4]{alv.mar-ine.svva19}).
\begin{lemma}
 \label{lm:alv.att}
Let the sequences $\{h_k\}$, $\{s_k\}$, $\{\alpha_k\}$ and $\{\delta_k\}$ in $[0,+\infty)$
and $\alpha\in \R$ be such that
$h_0=h_{-1}$, $0\leq \alpha_k\leq \alpha<1$ and
\begin{align}
  \label{eq:alv.att02}
h_{k+1}-h_k+s_{k+1}\leq \alpha_k(h_k-h_{k-1})+\delta_k\qquad \forall k\geq 0.
\end{align}
The following hold:
\begin{enumerate}
  \item [\emph{(a)}] For all $k\geq 1$,
 \begin{align}
       \label{eq:alv.att01}
       h_k+\sum_{j=1}^k\,s_j\leq
      h_0+\dfrac{1}{1-\alpha} \sum_{j=0}^{k-1}\,\delta_j.
    \end{align}
  \item [\emph{(b)}] If $\sum^{\infty}_{k=0}\delta_k <+\infty$, then $\lim_{k\to \infty}\,h_{k}$ exist, i.e., the sequence $\{h_k\}$ converges to some element in $[0,+\infty)$.
\end{enumerate}
\end{lemma}

\mgap

\begin{lemma}[Opial~\cite{opial-wea.bams67}]
 \label{lm:opial}
Let $\HH$ be a real Hilbert space, let $\emptyset \neq \mathcal{S}\subset \HH$ and let $\{p^k\}$ be a sequence in $\HH$
such that every weak cluster point of $\{p^k\}$ belongs to $\mathcal{S}$ and $\lim_{k\to
\infty}\,\norm{p^k-p}$ exists for every $p\in \mathcal{S}$. Then $\{p^k\}$
converges weakly to a point in $\mathcal{S}$.
\end{lemma}

\mgap

\begin{lemma} \emph{(\cite[Lemma A.2]{alv.eck.ger.mel-preprint19})}
 \label{lmm:inverse}
 The inverse function of the scalar map
\begin{align*}
 (0,2)\ni \beta \mapsto \dfrac{2(2-\beta)}{4-\beta+\sqrt{16\beta-7\beta^2}}\in (0,1)
\end{align*}
is given by
\begin{align*}
(0,1)\ni \overline{\alpha} \mapsto \dfrac{2(\overline{\alpha}-1)^2}{2(\overline{\alpha}-1)^2+3\overline{\alpha}-1}\in (0,2).
\end{align*}
\end{lemma}

\mgap

\begin{lemma}\emph{(\cite[Lemma A.3]{alv.eck.ger.mel-preprint19})}
 \label{lm:quadratic}
 Let $\R\ni \nu\mapsto q(\nu):=a \nu^2-b\nu+c$ be a real function and assume that $b,c>0$ and $b^2-4ac>0$. Define
\begin{align}
  \label{eq:root_b}
 \overline{\alpha}:= \dfrac{2c}{b+\sqrt{b^2-4ac}}>0.
\end{align}
\begin{itemize}
 \item[\emph{(i)}] If $a=0$, then $q(\cdot)$ is a decreasing affine function
 and $\overline{\alpha}>0$ as in \eqref{eq:root_b} is its unique root \emph{(}see
 \emph{Figure \ref{fig05}(a)}\emph{)}.
\item[\emph{(ii)}] If $a>0$ \emph{(}resp. $a<0$\emph{)}, then $q(\cdot)$ is a
 convex \emph{(}resp. concave\emph{)} quadratic function and $\overline{\alpha}>0$ as in
 \eqref{eq:root_b} is its smallest \emph{(}resp. largest\emph{)} root
 \emph{(}see \emph{Figure \ref{fig05}(b)} and \emph{Figure \ref{fig05}(c)},
 resp.\emph{)}.
\end{itemize}
In both cases $\emph{(i)}$ and $\emph{(ii)}$,  $\overline{\alpha}>0$ as in
\eqref{eq:root_b} is a root of $q(\cdot)$, and $q(\cdot)$ is decreasing in the
interval $[0,\overline{\alpha}]$ \emph{(}see \emph{Figure \ref{fig05}}\emph{)}.
\end{lemma}
\begin{figure}[H]
\centering
\begin{minipage}{0.3\textwidth}
\centering
\begin{tikzpicture}[domain=-0.5:2, scale=0.6]
    \draw[->,line width = 0.50mm] (-0.4,0) -- (3,0) node[right] {$\nu$};
    \draw[->,line width = 0.50mm] (0,-1.2) -- (0,4.2) node[above] {$q(\nu)$};
    \draw[color=red,line width = 0.50mm] (-0.4,3.5) -- (2.2,-1);
        \node[below left,black] at (0,0) {0};
        \node[below left,black] at (0,3) {c};
        \node[below left,blue] at (1.7,0) {$\overline{\alpha}$};
\end{tikzpicture} \\
(a) $a=0$
\end{minipage}
~
\begin{minipage}{0.3\textwidth}
\centering
\begin{tikzpicture}[domain=-0.3:3.6, scale=0.6]
    \draw[->,line width = 0.50mm] (-0.5,0) -- (4,0) node[right] {$\nu$};
    \draw[->,line width = 0.50mm] (0,-1.2) -- (0,4.2) node[above] {$q(\nu)$};
    \draw[red, line width = 0.50mm]   plot[smooth,domain=-0.3:3.6] (\x, {\x^2 -3.45*(\x)+2});
    \node[below left,black] at (0,0) {0};
    \node[below left,black] at (0,2.3) {c};
    \node[below left,blue] at (0.9,0) {$\overline{\alpha}$};
    \end{tikzpicture} \\
(b) $a>0$
\end{minipage}
~
\begin{minipage}{0.3\textwidth}
\centering
  \begin{tikzpicture}[domain=-0.5:4, scale=0.6]
    \draw[->,line width = 0.50mm] (-2.5,0) -- (2,0) node[right] {$\nu$};
    \draw[->,line width = 0.50mm] (0,-1.2) -- (0,4.2) node[above] {$q(\nu)$};
    \draw[color=red,line width = 0.50mm] (-2.2,-0.5) parabola[parabola height= 3.5cm] + (3.7,-0.3);
    \node[below left,black] at (0,0) {0};
    \node[above right,black] at (0,2.6) {c};
    \node[below left,blue] at (1.25,0) {$\overline{\alpha}$};
    \end{tikzpicture} \\
(c) $a<0$
\end{minipage}
  \caption{Possible cases for the real function $q(\cdot)$
           in Lemma \ref{lm:quadratic}.}
  \label{fig05}
\end{figure}
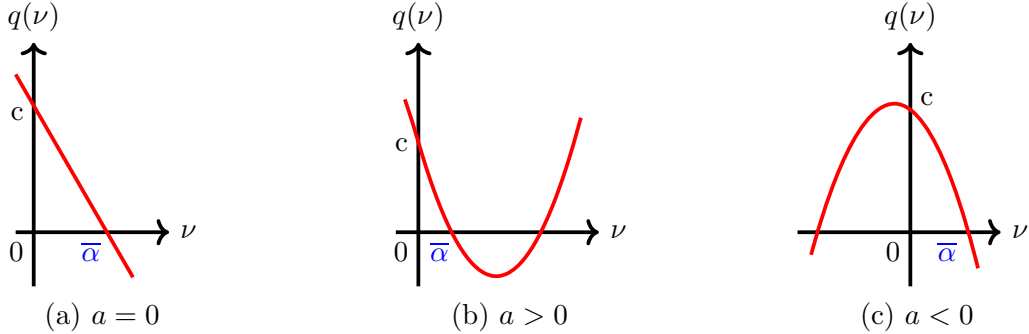

\mgap

The lemma below was proved (with a different notation) in \cite[Proposition 2.4]{alot.combet.shah.2014}.

\begin{lemma}
 \label{lm:comb}
  Let $\HH$ and $\mathcal{G}$ be real Hilbert spaces, let $A:\HH \rightrightarrows\HH$ and $B:\mathcal{G}\rightrightarrows \mathcal{G}$ be maximal monotone operators and let $G:\HH\rightarrow \mathcal{G}$ be a bounded linear operator.
  Let also $a^k\in A(r^k)$ and $b^k\in B(s^k)$ be such that $ r^k\rightharpoonup r^\infty$ and $b^k\rightharpoonup b^\infty$, for
  some $r^\infty\in \HH$ and $b^\infty\in \mathcal{G}$. If, $a^k + G^*b^k\to 0$ and $Gr^k-s^k\to 0$, then
  $b^\infty\in B(Gr^\infty)$ and $-G^* b^\infty\in A(r^\infty)$.
\end{lemma}


\def\cprime{$'$}

\end{document}